\documentclass[12pt,leqno]{amsart}
\usepackage{amsmath,amssymb,amsfonts}
\usepackage{eucal,enumerate}
\usepackage{xypic,graphicx,psfrag}
\usepackage{mathrsfs}
\usepackage{hyperref}

\usepackage{color} 

\usepackage{tikz}
\usetikzlibrary{decorations.pathreplacing}
\usepackage{scalefnt}

\newcommand\sectionpage{\newpage}	
\renewcommand\sectionpage{}	

\newtheorem{lem}{Lemma}[section]

\newtheorem{prop}[lem]{Proposition}
\newtheorem{thm}[lem]{Theorem}

\theoremstyle{definition}
\newtheorem{conj}[lem]{Conjecture}
\newtheorem{exam}[lem]{Example}
\newtheorem{problem}{Problem}

\numberwithin{equation}{section}
\numberwithin{table}{section}
\numberwithin{figure}{section}

\allowdisplaybreaks

\newcommand\vstrut[1]{\rule{0ex}{#1}}

\renewcommand\mod{\, \operatorname{mod}\, }
\renewcommand{\phi}{\varphi} 
\renewcommand{\epsilon}{\varepsilon}

\newcommand\cA{\mathscr{A}}		
\newcommand\cB{\mathcal{B}}
\newcommand\cBo{{\cB^\circ}}
\newcommand\cE{\mathcal{E}}
\renewcommand\cH{\mathcal{H}}	
\newcommand\cI{\mathcal{I}}
\newcommand\cP{\mathcal{P}}

\newcommand\cX{\mathcal{X}}
\newcommand\cY{\mathcal{Y}}
\newcommand\bbR{\mathbb{R}}
\newcommand\bbZ{\mathbb{Z}}
\newcommand\pN{\mathbb N}	
\newcommand\pP{\mathbb P}	
\newcommand\pQ{\mathbb Q}	

\newcommand\bz{\mathbf z}

\newcommand\lcm{\operatorname{lcm}}

\newcommand\M{\mathbf{M}}
\newcommand\Kot{Kot\v{e}\v{s}ovec}
\newcommand\cube{[0,1]^{2q}}
\newcommand\ocube{(0,1)^{2q}}

\allowdisplaybreaks

\begin{document}

\title{A $q$-Queens Problem \\
IV.  Attacking Configurations \\and Their Denominators}

\author{Seth Chaiken}
\address{Computer Science Department\\ The University at Albany (SUNY)\\ Albany, NY 12222, U.S.A.}
\email{\tt sdc@cs.albany.edu}

\author{Christopher R.\ H.\ Hanusa}
\address{Department of Mathematics \\ Queens College (CUNY) \\ 65-30 Kissena Blvd. \\ Queens, NY 11367-1597, U.S.A.}
\email{\tt chanusa@qc.cuny.edu}

\author{Thomas Zaslavsky}
\address{Department of Mathematical Sciences\\ Binghamton University (SUNY)\\ Binghamton, NY 13902-6000, U.S.A.}
\email{\tt zaslav@math.binghamton.edu}

\begin{abstract}
In Parts I--III we showed that the number of ways to place $q$ nonattacking queens or similar chess pieces on an $n\times n$ chessboard is a quasipolynomial function of $n$ whose coefficients are essentially polynomials in $q$. 

In this part we focus on the periods of those quasipolynomials.  We calculate denominators of vertices of the inside-out polytope, since the period is bounded by, and conjecturally equal to, their least common denominator.  We find an exact formula for that denominator of every piece with one move and of two-move pieces having a horizontal move.  For pieces with three or more moves, we produce geometrical constructions related to the Fibonacci numbers that show the denominator grows at least exponentially with $q$. 
\end{abstract}



\thanks{Version of \today.}
\thanks{The outer authors thank the very hospitable Isaac Newton Institute for facilitating their work on this project. The inner author gratefully acknowledges support from PSC-CUNY Research Awards PSCOOC-40-124, PSCREG-41-303, TRADA-42-115, TRADA-43-127, and TRADA-44-168.}

\maketitle
\vspace{-.3in}
\pagestyle{myheadings}
\markright{\textsc{A $q$-Queens Problem. IV. Configurations and Denominators}}\markleft{\textsc{Chaiken, Hanusa, and Zaslavsky}}

\tableofcontents

\sectionpage
\section{Introduction}\label{intro}

The famous $n$-Queens Problem asks for the number of arrangements of $n$ nonattacking queens---the largest possible number---on an $n\times n$\label{d:n} chessboard.\footnote{For surveys of the $n$-Queens Problem see, for instance, \cite{8Q, BS}.}  
There is no known general formula, other than the very abstract---that is to say impractical---one we obtained in Part~II and more recently the concrete but also impractical permanental formula of Pratt \cite{P}.  Solutions have been found only by individual analyses for small $n$.  

In this series of six papers \cite{QQs1, QQs2, QQs3, QQs5, QQs6}\footnote{Some of this paper was contained in the first version of Part~IV, announced in Parts~I and II, which is now mostly Part~V.}
we treat the problem by separating the board size, $n$, from the number of queens, $q$,\label{d:q} and rephrasing the whole problem in terms of geometry.  
We also generalize to every piece $\pP$ whose moves are, like those of the queen, rook, and bishop, unlimited in length.  Such pieces are known as ``riders'' in fairy chess (chess with modified rules, moves, or boards); an example is the nightrider,\label{N} whose moves are those of the knight extended to any distance.  
The problem then, given a fixed rider $\pP$, is:

\begin{problem}\label{Pr:formula}
Find an explicit, easily evaluated formula for $u_\pP(q;n)$\label{d:indistattacks}, the number of nonattacking configurations of $q$ unlabelled pieces $\pP$ on an $n \times n$ board.
\end{problem}

One would wish there to be a single style of formula that applies to all riders.  This ideal is realized to an extent by our work.  We proved in Part~I that for each rider $\pP$, $u_\pP(q;n)$ is a quasipolynomial function of $n$ of degree $2q$ and that the coefficients of powers of $n$ are given by polynomials in $q$, up to a simple normalization; for instance, the leading term is $n^{2q}/q!$.  Being a quasipolynomial means that for each fixed $q$, $u_\pP(q;n)$ is given by a cyclically repeating sequence of polynomials in $n$ (called the \emph{constituents} of the quasipolynomial); the shortest length of such a cycle is the \emph{period} of $u_\pP(q;n)$.  
That raises a fundamental question.

\begin{problem}\label{Pr:period}
What is the period $p$\label{d:p} of the quasipolynomial formula for $u_\pP(q;n)$?  (The period, which pertains to the variable $n$, may depend on $q$.)
\end{problem}

The period tells us how much data is needed to rigorously determine the complete formula; $2qp$ values of the counting function determine it completely, since the degree is $2q$ and the leading coefficient is known.  V\'aclav \Kot\ \cite{ChMathWeb, ChMath} has expertly used this computational approach to make educated guesses for many counting functions; proving the formula requires a proof of the period.
The difficulty with this computational approach is that, in general, $p$ is hard to determine and seems usually to explode with increasing $q$.  (We have reason to believe the period increases at least exponentially for any rider with at least three moves; see Theorem~\ref{T:exponential}.)  
A better way would be to find information about the $u_\pP(q;n)$ that is valid for all $q$.  For instance:

\begin{problem}\label{Pr:coeffs}
For a given piece $\pP$, find explicit, easily evaluated formulas for the coefficients of powers of $n$ in the quasipolynomials $u_\pP(q;n)$, valid for all values of $q$.
\end{problem}

A complete solution to Problem \ref{Pr:coeffs} would solve Problem \ref{Pr:formula}.  We think that is unrealistic but we have achieved some results.  
In Part~I we took a first step: in each constituent polynomial, the coefficient $\gamma_i$ of $n^{2q-i}$ is (neglecting a denominator of $q!$) itself a polynomial in $q$ of degree $2i$, that varies with the residue class of $n$ modulo a period $p_i$ that is independent of $q$.  In other words, if we count down from the leading term there is a general formula for the $i$th coefficient as a function of $q$ that has its own intrinsic period; the coefficient is independent of the overall period $p$.  
This opens the way to explicit formulas and in Part~II we found such a formula for the second leading coefficient as well as the complete quasipolynomial for an arbitrary rider with only one move---an unrealistic game piece but mathematically informative.  
In Part~III we found the third and fourth coefficients by concentrating on \emph{partial queens}, whose moves are a subset of the queen's.  

In the current part we push the same method in another direction by focusing on the effect of the number of moves.  
In our geometrical approach the boundaries of the square determine a hypercube in $\bbR^{2q}$\label{d:configsp} and the attack lines determine hyperplanes whose $1/(n+1)$-fractional points within the hypercube represent attacking configurations, which must be excluded; the nonattacking configurations are the $1/(n+1)$-fractional points inside the hypercube and outside every hyperplane, so it is they we want to count.  
The combination of the hypercube and the hyperplanes is an \emph{inside-out polytope} \cite{IOP}.  
We apply the inside-out adaptation of Ehrhart lattice-point counting theory, in which we combine by M\"obius inversion the numbers of lattice points in the polytope that are in each intersection subspace of the hyperplanes.  
The Ehrhart theory implies quasipolynomiality of the counting function and that the period divides the \emph{denominator} $D$,\label{d:D} defined as the least common multiple of the denominators of all coordinates of all vertices of the inside-out polytope.  
We investigate the denominators of individual vertices to provide a better understanding of the period. 
A consequence is an exact formula for the denominator of a one-move rider in Proposition~\ref{p:1moveIV}.  We apply the notion of trajectories from Hanusa and Mahankali \cite{Arvind} to prove a formula for the denominator of a two-move rider with a horizontal move in Proposition~\ref{P:denom2move}.  Theorem~\ref{T:exponential} proves that when a piece has three or more moves, by letting the number $q$ of pieces increase we obtain a sequence of inside-out polytope vertices with denominators that increase exponentially, and the polytope denominators may increase even faster.  These vertices arise from geometrical constructions related to Fibonacci numbers.

In general in Ehrhart theory the period and the denominator need not be equal and often are not, so it is surprising that in all our examples, and for any rider with only one move, they are.  We cannot prove that is always true for the inside-out polytopes arising from the problem of nonattacking riders, but this observation suggests that our work may be a good test case for understanding the relationship between denominators and periods. 

\medskip
A summary of this paper:  Section \ref{partI} recalls some essential notation and formulas from Parts~I--III, continuing on to describe the concepts we use to analyze the periods.
We turn to the theory of attacking configurations of pieces with small numbers of moves in Sections~\ref{sec:config1m}--\ref{sec:4move}, partly to establish formulas and conjectural bounds for the denominators of their inside-out polytopes, especially for partial queens and nightriders, and partly to support the exponential lower bound on periods and our many conjectures.  

We append a dictionary of notation for the benefit of the authors and readers.

\sectionpage\section{Essentials}\label{partI}

\subsection{From before}\

Each configuration-counting problem arises from three choices: a chess board, a chess piece, and the number of pieces.  (The size of the board is considered a variable within the problem.)  

The board $\cB$,\label{d:cB} is any rational convex polygon, i.e., it has rational corners.  (We call the vertices of $\cB$ its \emph{corners} to avoid confusion with other points called vertices.)
The pieces are placed on the integral points in the interior of an integral dilation of $\cB$.  In many cases we will consider our board to be the unit square $\cB=[0,1]^2$, in that the $n\times n$ chessboard corresponds to the interior points of the $(n+1)$-dilation of $\cB$.  We call the open or closed unit square the ``(square) board''; it will always be clear which board we mean.  

The piece $\pP$\label{d:P} has \emph{moves} defined as all integral multiples of a finite set $\M$\label{d:moveset} of non-zero, non-parallel integral vectors $m = (c,d) \in \bbZ^2$,\label{d:mr} 
which we call the \emph{basic moves}.  Each one must be reduced to lowest terms; that is, its two coordinates need to be relatively prime; and no basic move may be a scalar multiple of any other.  Thus, the slope of $m$ contains all  necessary information and can be specified instead of $m$ itself.  
We say two distinct pieces \emph{attack} each other if the difference of their locations is a move.  
In other words, if a piece is in position $z:=(x,y) \in \bbZ^2$, it attacks any other piece in the lines $z + r m$ for $r \in \bbZ$ and $m \in \M$.  For example, the set $\M$ is $\{(1,1),(1,-1)\}$ for the bishop.  
Attacks are not blocked by a piece in between (a superfluous distinction for nonattacking configurations), and they include the case where two pieces occupy the same location.
The number $q$ is the number of pieces that are to occupy places on the board; we assume $q > 0$.  

A \emph{configuration} $\bz=(z_1,\ldots,z_q)$\label{d:config} is any choice of locations for the $q$ pieces, including on the board's boundary, where $z_i:=(x_i,y_i)$ denotes the position of the $i$th piece $\pP_i$.  
(The boundary, while not part of the board proper, is necessary in our counting method in order to estimate periods.)  
That is, $\bz$ is an integral point in the $(n+1)$-fold dilation of the $2q$-dimen\-sion\-al closed, convex polytope $\cP=\cB^q$.\label{d:cP}.  
If we are considering the undilated board, $\bz$ is a fractional point in $\cB^q$.  We consider these two points of view equivalent; it will always be clear which kind of board or configuration we are dealing with.  
Any integral point $\bz$ in the dilated polytope, or its fractional representative $\frac{1}{n+1}\bz$ in the undilated board, represents a placement of pieces on the board, and vice versa; thus we use the same term ``configuration'' for the point and the placement.

The constraint for a \emph{nonattacking configuration} is that the pieces must be in the board proper (so $\bz\in(\cBo)^q$ or its dilation) and that no two pieces may attack each other.  In other words, if there are pieces at positions $z_i$ and $z_j$, then $z_j-z_i$ is not a multiple of any $m\in\M$; equivalently, $(z_j-z_i)\cdot m^\perp \neq 0$ for each $m\in\M$, where $m^\perp := (d,-c)$\label{d:mrperp}.  

For counting we treat nonattacking configurations as ``interior'' integral lattice points in the dilation of an inside-out polytope $(\cP,\cA_\pP)$, where $\cP=\cB^q$ and $\cA_\pP$ is the \emph{move arrangement}\label{d:AP}, whose members are the \emph{move hyperplanes} (or \emph{attack hyperplanes}) 
$$
\cH^{d/c}_{ij} := \{ (z_1,\ldots,z_q) \in \bbR^{2q} : (z_j - z_i) \cdot m^\perp = 0 \}
\label{d:slope-hyp}
$$ 
for $m=(c,d)\in\M$; the equations of these hyperplanes are called the \emph{move equations} (or \emph{attack equations}) of $\pP$.  
We supplement this notation with 
$$\cX_{ij}:=\cH_{ij}^{1/0}: x_i=x_j \ \text{ and } \ \cY_{ij}:=\cH_{ij}^{0/1}: y_i=y_j\label{d:XY}$$ 
for the hyperplanes that express an attack along a row or column.  
Thus, by the definition of the interior of an inside-out polytope \cite{IOP}, a configuration $\bz\in\cP$ is nonattacking if and only if it is in $\cP^\circ$ and not in any of the hyperplanes $\cH^{d/c}_{ij}$.  
A \emph{vertex} of $(\cP,\cA_\pP)$ is any point in $\cP$ that is the intersection of facets of $\cP$ and hyperplanes of $\cA_\pP$.  For instance, it may be a vertex of $\cP$, or it may be the intersection point of hyperplanes if that point is in $\cP$, or it may be the intersection of some facets and some hyperplanes.

The $2q$ equations that determine a vertex $\bz$ of $(\cB^q,\cA_\pP)$ are {\em move equations}, associated to hyperplanes $\cH^{d/c}_{ij} \in \cA_\pP$, and \emph{boundary equations} or {\em fixations}, of the form $z_i \in$ an edge line $\cE$\label{d:cE} of $\cB$.  
(For the square board a fixation is one of $x_i=0$, $y_i=0$, $x_i=1$, and $y_i=1$.  In a dilation $N\cdot\cB$, a fixation has the form $z_i \in N\cdot\cE$; on the square board, $x_i=0$, $y_i=0$, $x_i=N$, or $y_i=N$.)  
The vertex $\bz$ represents a configuration with $\pP_i$ on the boundary of the board (if it has a fixation) or attacking one or more other pieces (if in a hyperplane).  We call the configuration of pieces that corresponds to a vertex $\bz$ a \emph{vertex configuration}.

\subsection{Periods and denominators}\label{sec:period}\

Let us write $\Delta(\bz)$\label{d:Deltabz} for the least common denominator of a fractional point $\bz \in \bbR^{2q}$ and call it the \emph{denominator of\/ $\bz$}.  
The denominator $D=D(\cB^q,\cA_\pP)$ of the inside-out polytope is then the least common multiple of the denominators $\Delta(\bz)$ of the individual vertices.  

One way to find $\Delta(\bz)$ for a vertex $\bz$ is to find its coordinates by intersecting move hyperplanes of $\pP$ and facet hyperplanes of $\cB^q$.
There is an equivalent method to find $\Delta(\bz)$.  For a set of move equations and fixations producing a vertex configuration $\bz$, notice that the $\Delta(\bz)$-multiple of $\bz$ has integer coordinates and no smaller multiple of $\bz$ does.  This proves:

\begin{lem}\label{L:DeltaN}
For a vertex $\bz$ of $(\cB^q, \cA_\pP)$, $\Delta(\bz)$ equals the smallest integer $N$ such that a configuration $N\cdot\bz$ satisfying the move equations and fixations $z_i\in N\cdot\cE$ for edge lines $\cE$ has integral coordinates.  
\end{lem}

We expect the period to be weakly increasing with $q$ and also with the set of moves; that is, if $q'>q$, the period for $q'$ pieces should be a multiple of that for $q$; and if $\pP'$ has move set containing that of $\pP$, then the period of $\pP'$ should be a multiple of that of $\pP$.  We cannot prove either property, but they are obvious for denominators, and we see in examples that the period equals the denominator.  Write $D_q(\pP)$\label{d:DqP} for the denominator of the inside-out polytope $(\cube,\cA_\pP^q)$.  (The optional superscript in $\cA_\pP^q$ is for when the number of pieces varies.)

\begin{prop}\label{P:Dincrease}
Let $\cB$ be any board, let $q'>q>0$, and suppose $\pP$ and $\pP'$ are pieces such that every basic move of $\pP$ is also a basic move of $\pP'$.  Then the denominators satisfy $D_q(\pP) | D_{q'}(\pP)$ and $D_q(\pP) | D_q(\pP')$.
\end{prop}

\begin{proof}
The first part is clear if we embed $\bbR^{2q}$ into $\bbR^{2q'}$ as the subspace of the first $2q$ coordinates, so the polytope $\cB^q$ is a face of $\cB^{2q'}$ and the move arrangement $\cA_\pP^q$ in $\bbR^{2q}$ is a subarrangement of the arrangement induced in $\bbR^{2q}$ by $\cA_\pP^{q'}$.

The second part is obvious since $\cA_\pP^q \subseteq \cA_{\pP'}^{q}$.
\end{proof}

A \emph{partial queen} $\pQ^{hk}$ \label{d:partQ} is a piece with $h$ basic moves that are horizontal or vertical (obviously, $h\leq2$) and $k$ basic moves at $\pm45^\circ$ to the horizontal (also, $k\leq2$). Four of them are the semirook $\pQ^{10}$, rook $\pQ^{20}$, semibishop $\pQ^{01}$, anassa $\pQ^{11}$,\footnote{``Anassa'' is archaic Greek feminine for a tribal chief, i.e., presumably for the consort of a chief \cite{Anax}.} and semiqueen $\pQ^{21}$.  Anticipating later results, we have the striking conclusion that:

\begin{thm}\label{prop:denom1}
On the square board there are only four pieces whose denominator is $1$ for all $q\geq 1$.  They are the rook, semirook, semibishop, and anassa. Their counting functions are polynomials in $n$.
\end{thm}

\begin{proof}
By Theorem~\ref{T:maxcd} the only one-move pieces with denominator 1 are the semirook and semibishop, so by Proposition~\ref{P:Dincrease}  the only pieces that can have denominator 1 are  partial queens.  Theorem~\ref{T:exponential} reduces that to pieces with two or fewer moves.  The rook obviously has denominator 1, while the bishop has denominator and period 2 when $q\geq 3$ (Part VI). Proposition~\ref{prop:anassa} implies that the anassa has denominator 1. 
\end{proof}

Although in Ehrhart theory periods often are less than denominators, we observe that not to be true for our solved chess problems.  We believe that will some day become a theorem.

\begin{conj}[Conjecture II.8.6] \label{Cj:p=D}
For every rider piece $\pP$ and every number of pieces $q\geq1$, the period of the counting quasipolynomial $u_\pP(q;n)$ equals the denominator $D(\cube,\cA_\pP)$ of the inside-out polytope for $q$ pieces $\pP$.
\end{conj}

If Conjecture~\ref{Cj:p=D} is true, the four special partial queens of Theorem~\ref{prop:denom1} will be the only pieces with period 1.

\begin{conj}\label{Cj:period1}
The rook, semirook, semibishop, and anassa are the only four pieces whose counting functions on the square board are polynomials in $n$. 
\end{conj}

\sectionpage
\section{One-move riders}\label{sec:config1m}

The denominator of the inside-out polytope of a one-move rider on  an arbitrary board $\cB$ can be explicitly determined.

Given a move $m=(c,d)$, the line parallel to $m$ through a corner $z$ of $\cB$ may pass through another point on the boundary of $\cB$.  Call that point the {\em antipode} of $z$.  The antipode may be another corner of $\cB$.  When $m$ is parallel to an edge $z_iz_j$ of $\cB$, we consider $z_i$ and $z_j$ to be each other's antipodes.

\begin{prop}\label{p:1moveIV}
For a one-move rider $\pP$ with move $(c,d)$, the denominator of the inside-out polytope $(\cB^q,\cA_{\pP})$ equals the least common denominator of the corners of $\cB$ when $q=1$, and when $q\geq 2$ it equals the least common denominator of the corners of $\cB$ and their antipodes.
\end{prop}

\begin{proof}
A vertex of $(\cB^q,\cA_{\pP})$ is generated by some set of hyperplanes, possibly empty, and a set of fixations. The total number of hyperplanes and fixations required is $2q$. When $q=1$, because there are no move equations involved, a vertex of the inside-out polytope is a corner of $\cB$.   

When $q\geq 2$, a vertex is determined by its fixations and the intersection $\cI$ of the move hyperplanes it lies in.  Let $\pi$ be the partition of $[q]$ into blocks for which $i$ and $j$ are in the same block if $\cH^{d/c}_{ij}$ is one of the hyperplanes containing $\cI$. The number of hyperplanes necessary to determine $\cI$ is $q$ minus the number of blocks of $\pi$.  ($\cI$ will be contained in additional, unnecessary hyperplanes if a block of $\pi$ has three or more members; we do not count those.)  

Consider a particular block of $\pi$, which we may suppose to be $[k]$ for some $k\geq1$.  We need $k+1$ fixations on the $k$ pieces to specify a vertex, so there must be two fixations that apply to the same $i\in[k]$, anchoring $z_i$ to a corner of $\cB$.  The remaining $k-1$ fixations fix the other values $z_j$ for $j\in[k]$ to either $z_i$'s corner or its antipode.

It follows that all vertices $(z_1,\hdots,z_q)$ of the inside-out polytope satisfy that each $z_i$ is either a corner or a corner's antipode for all $i$.  Furthermore, with at least two pieces and for every corner $z$, it is possible to create a vertex containing $z$ and its antipode as components, from which the proposition follows.
\end{proof}

For the square board, the corners are $(0,0)$, $(1,0)$, $(0,1)$, and $(1,1)$, and the antipodes have denominator $\max(|c|,|d|)$. Proposition~II.6.2 tells us that the coefficient of $n^{2q-3}$ has period $\max(|c|,|d|)$ when $q\geq 2$.  Thus:

\begin{thm}\label{T:maxcd}
On the square board with $q\geq2$ copies of a one-move rider with basic move $(c,d)$, the period of $u_\pP(q;n)$ is $\max(|c|,|d|)$.
\end{thm}

This theorem was previously Conjecture II.6.1.  Hence Conjecture~\ref{Cj:p=D} is true for one-move riders: the period agrees with the denominator.

\sectionpage
\section{Two-move riders}\label{sec:config2m}

The denominator of the inside-out polytope for a two-move rider is also well understood.  Hanusa and Mahankali \cite{Arvind} extend the concept of antipode to define configurations of two-move riders called trajectories.  

An \emph{infinite trajectory} is an infinite sequence $[z_1,z_2,z_3,\hdots]$ of points on the boundary of $\cB$ such that the points $z_i$ and $z_{i+1}$ are related by an attack equation.  
We only need trajectories where the attack equations alternate between those of the moves $m_1$ and $m_2$.  
A \emph{trajectory} (or \emph{finite trajectory}) $T=[z_1,\hdots,z_l]$ is an initial sequence of an infinite trajectory, with \emph{length} $l\geq 1$.  It is subject to the following stopping rule:  it may stop at any step $l\geq1$, except that it is forced to stop if it hits an edge of $\cB$ that is parallel to a move vector or if it repeats the first point (then $z_l=z_1$).  
A trajectory that stops where it is forced to is \emph{maximal}; every trajectory is therefore an initial portion of a maximal trajectory.
A trajectory of length $1$ is \emph{trivial}.  
The \emph{extension} of $T$ is the trajectory $\widehat{T}=[z_1,\hdots,z_l,z_{l+1}]$ involving one more point from its infinite trajectory.

A \emph{corner trajectory} is a trajectory that includes a corner of $\cB$.  A \emph{rigid cycle} is a trajectory that does not contain a corner, that returns to its initial point, and whose system of attack equations and fixation equations is linearly independent.  A point in the interior where two extended trajectories cross or an extended trajectory crosses itself is called a {\em crossing point}. With these definitions, we state the main theorem of \cite{Arvind}.


\begin{thm}[{\cite[Theorem~4.10]{Arvind}}]\label{T:Arvind}
The denominator of $(\cB^q,\cA_{\pP}^q)$ is equal to the least common multiple of
\begin{enumerate}
\item the denominators of points in the boundary of $\cB$ belonging to corner trajectories and rigid cycles of length at most $q$, and
\item the denominators of all crossing points $\mathbf{c}$ of corner trajectories and rigid cycles, where if $\mathbf{c}$ is a self-crossing, the length of its defining trajectory is at most $q-1$, and if $\mathbf{c}$ is a crossing of two distinct
trajectories, their lengths must sum to at most $q-1$.
\end{enumerate}
\end{thm}

Figure~\ref{fig:twomove} shows that trajectories can involve complex dynamics and denominators that grow quickly.  The pattern of piece placements depends where the slopes fall (less than $-1$, between $-1$ and $0$, between $0$ and $1$, or greater than $1$).    
\begin{figure}[htbp]
	\includegraphics[width=1.9in]{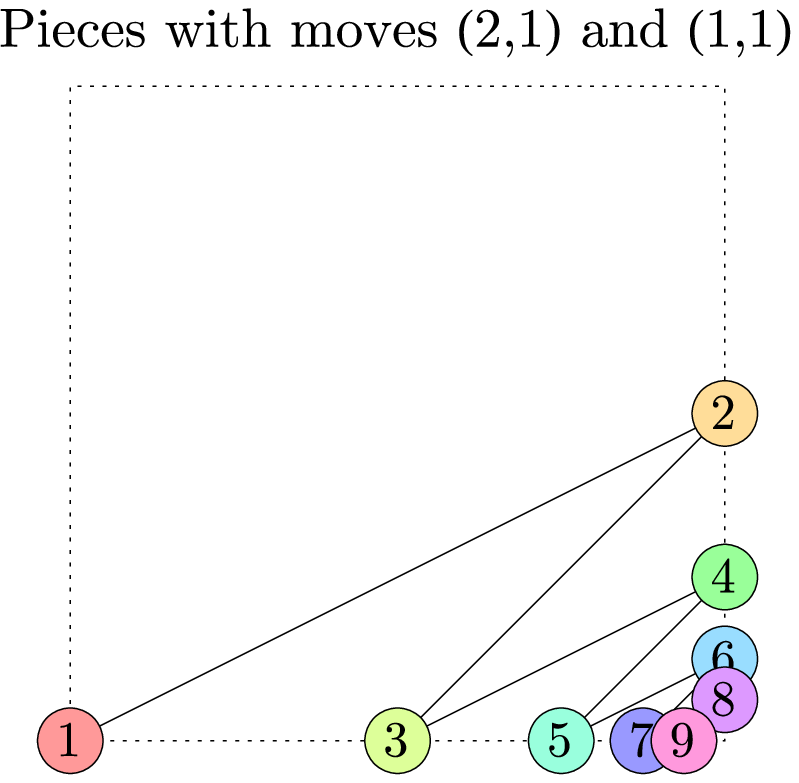}\quad\quad
	\includegraphics[width=1.9in]{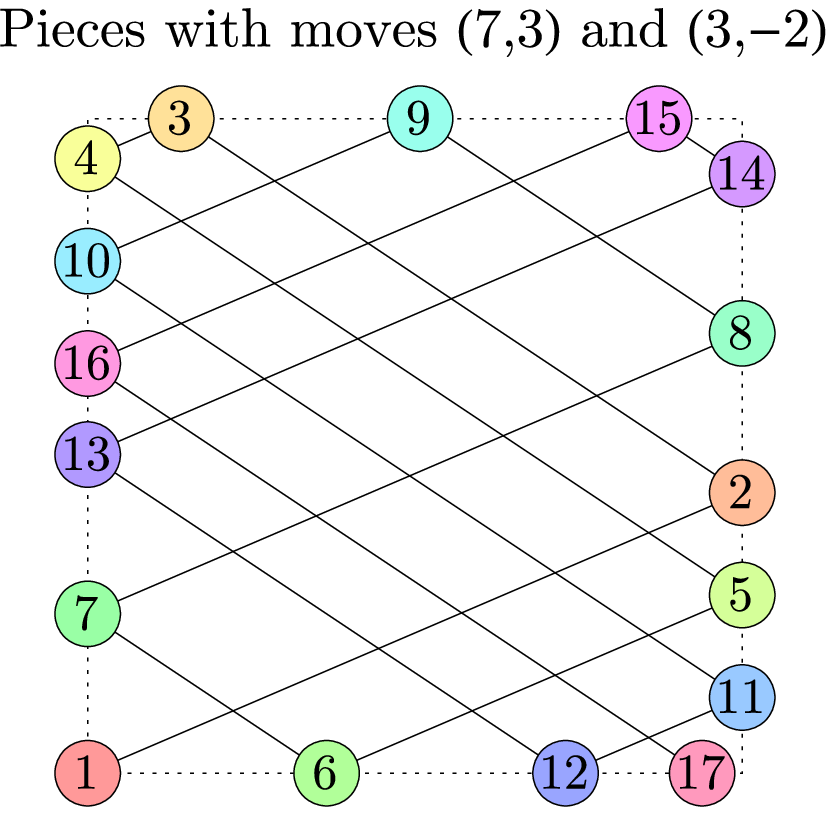}\quad\quad
	\includegraphics[width=1.9in]{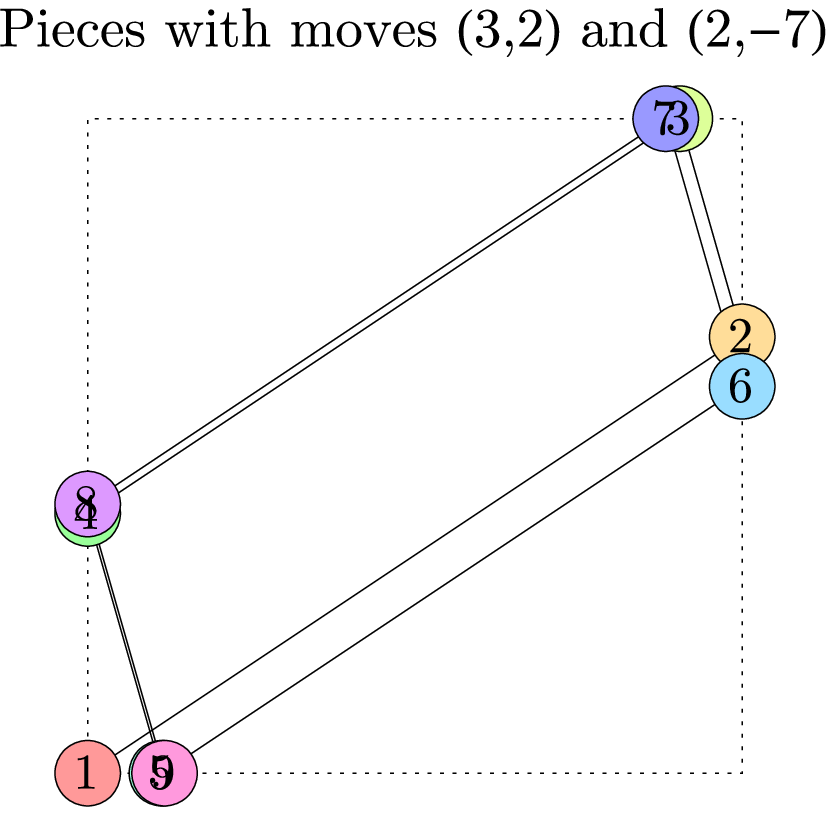}
	\caption{A two-move rider with diagonal slopes can produce configurations with arbitrarily large denominators.  The coordinates of the ninth pieces are, from left to right, $(15/16, 0)$, $(32/63, 1)$, and 
		$(22850/194481, 0)$.
	}
	\label{fig:twomove}
\end{figure}
In many cases, the piece positions and the denominator of the corresponding vertex follow a pattern that we found difficult to describe completely. 

For certain two-move riders it is possible to state the denominator explicitly.

\begin{prop}\label{P:denom2move}
Let $\cB$ be the square board and let $c$ and $d$ be relatively prime positive integers. The denominator of the inside-out polytope for $q$ two-move riders with moves $(1,0)$ and $(\pm c,\pm d)$ is
\begin{equation*}
D=
\begin{cases}
1 & \text{if } q=1, \\
d & \text{if } d \geq c \text{ and } q>1, \\
c & \text{if } d < c \text{ and } 1 < q \leq 2\lfloor c/d \rfloor +1, \\
cd & \text{if } d < c \text{ and } q \geq 2\lfloor c/d \rfloor +2. 
\end{cases}
\end{equation*}
\end{prop}

\begin{proof}
We consider the two-move rider with basic move $(+c,+d)$; the other signs follow by reflecting the board.  We assume $q>1$ since otherwise the denominator trivially equals 1. There are no self-crossing trajectories nor, by the stopping rule, rigid cycles.  
A corner trajectory with basic move $(1,0)$ can only include a corner if it begins (or ends---equivalently) there since it can only reach or leave a corner by the move $(c,d)$.
The only two maximal corner trajectories, $T$ and $T'$, start at $(0,0)$ and $(1,1)$, respectively; they are $180^\circ$ rotations of each other about the center of the square.  

If $d\geq c$, then $T=[(0,0),(c/d,1)]$ is the maximal trajectory that starts at $(0,0)$.  It does not cross $T'$, so since $q\geq2$ it contributes a factor $d$ to the denominator $D$, as does $T'$.  Therefore $D=d$.

Assume $d<c$ from now on.  The trajectories zigzag back and forth across the square.  $T$ starts at $z_1=(0,0)$ and visits the points $z_2=(1,d/c)$, $z_3=(0,d/c)$, $z_4=(1,2d/c)$, and so forth until it stops.  If it is long enough, $T$ continues until it reaches $y=1$, where it must stop.  

When $d=1$ all points on $T$ lie on the maximal corner trajectory from $(0,0)$, which ends at $(1,1)$, so corner trajectories contribute a factor $c$ to the denominator $D$ (since $q\geq2$).  Consequently, $D=c$.

When $d>1$ and $q$ is large enough, the zigzag pattern can continue up to $z_{2k+1} = (0,kd/c)$, where $k=\lceil c/d \rceil-1=\lfloor c/d \rfloor$, so $z_{2k+2}$ is located along the line $y=1$ with $x$-coordinate $c/d-\lfloor c/d \rfloor$.  (See Figure~\ref{fig:pfdenom2move}.)  A trajectory from $(0,0)$ can continue to $z_{2k+2}$ if and only if $q\geq 2\lfloor c/d \rfloor+2$.  
Hence there is a corner trajectory $T$ contributing $c$ to $D$ when $q\geq 2$, and there is one contributing $d$ and $c$ when $q\geq 2\lfloor c/d \rfloor+2$ (so that $T$ can be chosen maximal). By central symmetry, the points along $T$ and $T'$ have the same denominators, so $T'$ contributes nothing new.  
\begin{figure}[htbp]
\includegraphics[width=2in]{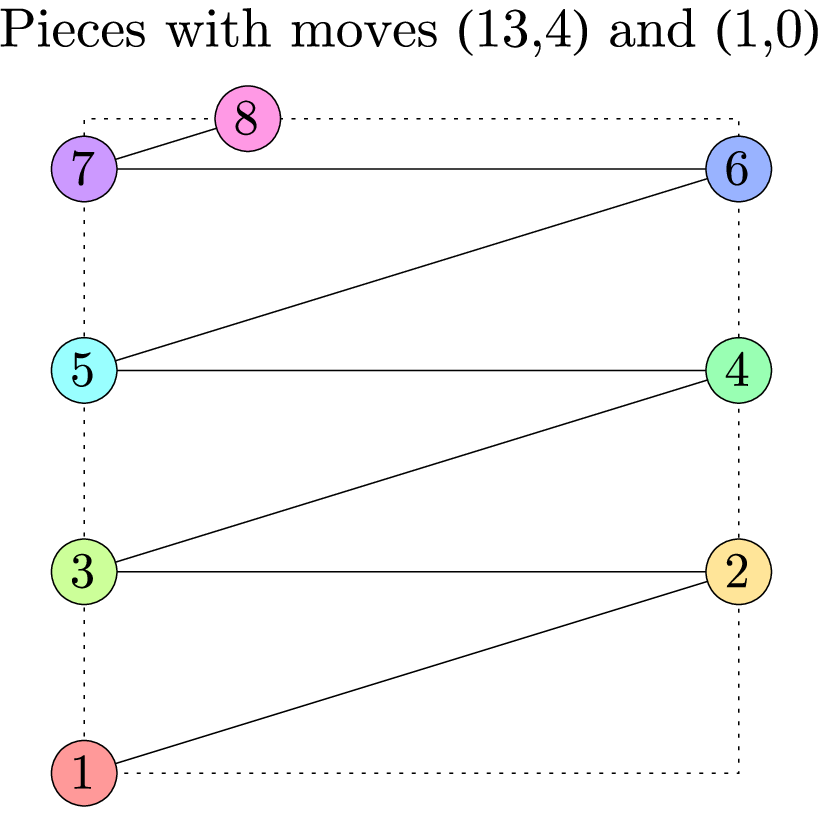}
\caption{
In a trajectory of a two-move rider with a horizontal move and $d<c$, the coordinates of pieces have denominator $c$ until a piece has $y$-coordinate $1$.  In this example $(c,d)=(13,4)$ and the coordinates of the even pieces are $\pP_2(1,4/13)$, $\pP_4(1,8/13)$, $\pP_6(1,12/13)$, and $\pP_8(1/4,1)$.
}
\label{fig:pfdenom2move}
\end{figure}

It remains to calculate the crossing points of $\widehat{T}$ (of length $l+1$) and $\widehat{T'}$ (of length $l'+1$), with the restriction that $l+l' \leq q-1$ (so there are enough pieces to occupy $T$, $T'$, and the crossing).  
Intersections occur when a sloped edge of one extended trajectory---say $\widehat{T}$, by symmetry---intersects a horizontal edge of the other trajectory, $\widehat{T'}$.  The sloped edges of $\widehat{T}$ join $(0,jd/c)$ to $(1,(j+1)d/c)$ for $0\leq j\leq \lfloor (l-1)/2\rfloor$.  The horizontal edges of $\widehat{T'}$ occur at $y$-coordinates $1-id/c$ for $1\leq i \leq \lfloor l'/2\rfloor$.   The $x$-coordinate of an intersection point of this type is $c/d-\lfloor c/d \rfloor$, whose denominator is $d$.  
When $q \geq 2\lfloor c/d\rfloor +2$, $d$ already appears in $D$ from a maximal corner trajectory; therefore a crossing contributes nothing new in that case.

On the other hand, when $q < 2\lfloor c/d\rfloor +2$ there is no crossing.  
For there to be an intersection, the sloped edge must end at a greater $y$-coordinate than that of the intersecting horizontal edge.  Hence an intersection will occur only if 
$\lfloor (l-1)/2\rfloor+\lfloor l'/2\rfloor+1>c/d$; equivalently,
\begin{equation*}\label{eq:intersection}
\bigg\lceil \frac{l}{2} \bigg\rceil + \bigg\lfloor \frac{l'}{2} \bigg\rfloor > \frac{c}{d}.
\end{equation*}
However, when $l+l'\leq q-1$, which is $\leq 2\lfloor c/d\rfloor$, 
\begin{equation*}
\bigg\lceil \frac{l}{2} \bigg\rceil + \bigg\lfloor \frac{l'}{2} \bigg\rfloor \leq 
\bigg\lceil \frac{l}{2} \bigg\rceil + \bigg\lfloor \frac{c}{d} \bigg\rfloor + \bigg\lfloor \frac{-l}{2} \bigg\rfloor = 
\bigg\lfloor \frac{c}{d} \bigg\rfloor < \frac{c}{d}.
\end{equation*}
Therefore no crossing points exist for $q\leq 2\lfloor c/d\rfloor +1$.

It follows that the overall denominator in the trajectories is $c$ if $q \leq 2\lfloor c/d \rfloor +1$ and $cd$ if $q \geq 2\lfloor c/d \rfloor +2$.
\end{proof}

Proposition~\ref{P:denom2move} applies to the anassa. 

\begin{prop}\label{prop:anassa}
On the square board, the denominator and period of the anassa are $1$.
\end{prop}

\sectionpage
\section{Pieces with Three or More Moves}\label{sec:moremoves}

From now on we assume the square board.  

With three or more moves, three new configurations appear: a triangle of pairwise attacking pieces (Section~\ref{sec:config3m}), a golden parallelogram (Example~\ref{ex:parallelograms}), and with four moves, a twisted Fibonacci spiral (Example~\ref{ex:twisted}).  
The latter two, which combine $2q-3$ move equations and three fixations, yield the largest vertex denominators known to us (see especially Section \ref{sec:config3p}). 

A piece with at least three moves has not only new configuration types; it also enters a new domain of complexity.  There is no straightforward generalization of Theorem~\ref{T:Arvind} involving rigid cycles, like that in Figure~\ref{fig:threemoveNEW}(a), whose points are all on the boundary of the board.  We now see rigid cycles of a new kind, as in Figure~\ref{fig:threemoveNEW}(b), whose points are not all boundary points.  We also lose constancy of the denominator in a very strong way.  For every piece with at least three moves, the denominator grows exponentially or faster with $q$ (Theorem~\ref{T:exponential}); and if Conjecture~\ref{Cj:p=D} is true, the period of its counting quasipolynomial grows as fast.  


\subsection{Triangle configurations}\label{sec:config3m}\

With three (or more) moves, a new key configuration appears: a triangle of pairwise attacking pieces, whose denominator we can calculate exactly.

Consider a piece with the three basic moves $m_1=(c_1,d_1)$, $m_2=(c_2,d_2)$, and $m_3=(c_3,d_3)$. Since no move is a multiple of another, there exist nonzero integers $w_1$, $w_2$, and $w_3$ \label{d:w} with $\gcd(w_1,w_2,w_3)=1$ such that $w_1m_1+w_2m_2+w_3m_3=(0,0)$.  The $w_i$ are unique up to negating them all.  

\begin{prop}\label{P:triangular}
For $q=3$, a triangular configuration of three pieces on the square board, attacking pairwise along three distinct move directions $m_1=(c_1,d_1)$, $m_2=(c_2,d_2)$, and $m_3=(c_3,d_3)$, together with three fixations that fix its position in the square $[0,1]^2$, gives a vertex $\bz$ of the inside-out polytope.  Its denominator is 
\begin{equation}
\label{E:triangular}
\Delta(\bz)=\max(\lvert w_1c_1\rvert, \lvert w_1d_1\rvert, \lvert w_2c_2\rvert,\lvert w_2d_2\rvert,  \lvert w_3c_3\rvert, \lvert w_3d_3\rvert).
\end{equation}
\end{prop}

The pieces may be at corners, and there may be two pieces on the same edge.  The three fixations may be choosable in more than one way but they will give the same denominator.

\begin{proof}
There is a unique similarity class of triangles with edge directions $m_1$, $m_2$, and $m_3$, if we define triangles with opposite orientations to be similar.  We can assume the the pieces are located at coordinates $z_1,z_2,z_3$ with $\max y_i - \min y_i \leq \max x_i - \min x_i$ (by diagonal reflection) $= x_3-x_1$ since we can assume $x_1 \leq x_2 \leq x_3$ (by suitably numbering the pieces), with $y_1 \leq y_3$ (by horizontal reflection), and with $z_2$ below the line $z_1z_3$ (by a half-circle rotation).  The reflections change the move vectors $m_i$ by negating or interchanging components; that makes no change in Equation~\eqref{E:triangular}.  We number the slopes so that $m_1$, $m_2$, and $m_3$ are, respectively, the directions of $z_1z_2$, $z_1z_3$, and $z_2z_3$.

Given these assumptions the triangle must have width $x_3-x_1=1$, since otherwise it will be possible to enlarge it by a similarity transformation while keeping it in the square $[0,1]^2$; consequently $x_1=0$ and $x_3=1$.  
Furthermore, the slopes satisfy $d_1/c_1 < d_2/c_2 < d_3/c_3$.  (If $c_3=0$ we say the slope $d_3/c_3=+\infty$ and treat it as greater than all real numbers.  If $c_1=0$ we say $d_1/c_1=-\infty$ and treat it as less than all real numbers.  $c_2$ cannot be 0.)  Our configuration has $d_2/c_2 \geq 0$ so two slopes are nonnegative but $d_1/c_1$ may be negative.  That gives two cases.

If $d_1/c_1 \leq 0$, we choose fixations $x_1=0$, $y_2=0$, and $x_3=1$.  
(A different choice of fixations is possible if $z_1z_2$ is horizontal or vertical, if $z_2z_3$ is horizontal or vertical, or if  $z_1z_3$ is horizontal, not to mention combinations of those cases.  Note that the denominator computation depends on the differences of coordinates rather than their values.  In each horizontal or vertical case the choice of fixations affects only the triangle's location in the square, not its size or orientation.)

If $d_1/c_1 > 0$, we choose fixations $x_1=y_1=0$ and $x_3=1$.  

The rest of the proof is the same for both cases.  First we prove that the configuration is a vertex.  That means the locations of the three pieces are completely determined by the fixations and the fact that $\bz=(z_1,z_2,z_3) \in \cH_{12}^{m_1}\cap\cH_{13}^{m_2}\cap\cH_{23}^{m_3}$.  We know the similarity class of $\triangle z_1z_2z_3$ and its orientation.  The fixations of $\pP_1$ and $\pP_3$ determine the length of the segment $z_1z_3$.  That determines the congruence class of $\triangle z_1z_2z_3$, and the fixations determine its position.  Thus, $\bz$ is a vertex.

We now aim to find the smallest integer $N$ such that $N\cdot\triangle z_1z_2z_3$ embeds in the integral lattice $[0,N]\times[0,N]$, i.e., it has integral coordinates.  
By the definition of $w_1$, $w_2$, and $w_3$, we know that $z_1'=(0,0)$, $z_2'=-w_1m_1$, and $z_3'=w_2m_2$ gives an integral triangle that is similar to $\triangle z_1z_2z_3$ and similarly or oppositely oriented, because its sides have the same or opposite slopes.  If $\triangle z_1'z_2'z_3'$ is oppositely oriented to $\triangle z_1z_2z_3$ (that means $z_2'$ is above the line $z_1'z_3'$), we can make the orientations the same by negating all $w_i$.  
Given these restrictions $\triangle z_1'z_2'z_3'$ is as compact as possible, for if some multiple $\nu \triangle z_1'z_2'z_3'$ were smaller ($0<\nu<1$) and integral, then $(\nu w_1)m_1+(\nu w_2)m_2+(\nu w_3)m_3=0$ with integers $\nu w_1, \nu w_2, \nu w_3$, so $\nu$ would be a proper divisor of 1.  By Lemma~\ref{L:DeltaN}, $N=\Delta(\bz)$.  
We can now translate $\triangle z_1'z_2'z_3'$ to the box $[0,N]\times[0,N]$ where 
\[
N=\max(\lvert w_1c_1\rvert, \lvert w_1d_1\rvert, \lvert w_2c_2\rvert,\lvert w_2d_2\rvert,  \lvert w_3c_3\rvert, \lvert w_3d_3\rvert).
\qedhere
\]
\end{proof}

\begin{exam}\label{ex:threemove}
The three-move partial nightrider has move set \(\M=\{(2,-1),(2,1),(1,2)\}\).  Because $3\cdot(2,-1)-5\cdot(2,1)+4\cdot(1,2)=(0,0)$ the denominator of its triangle configuration is 
\[\max(\lvert 6\rvert, \lvert -3\rvert, \lvert -10\rvert,  \lvert -5\rvert, \lvert 4\rvert, \lvert 8\rvert)=10,\] 
as shown in Figure~\ref{fig:threemoveNEW}(a).  

For the piece with move set \(\M=\{(1,2),(3,1),(4,3)\}\), the denominator of its triangle configuration is $4$ because the moves satisfy $(1,2)+(3,1)-(4,3)=(0,0)$.  Furthermore, all move slopes are positive, so the configuration does not entirely lie on the boundary of $\cB$, as shown in Figure~\ref{fig:threemoveNEW}(b).
\end{exam}

\begin{figure}[htbp]
\includegraphics[height=1.7in]{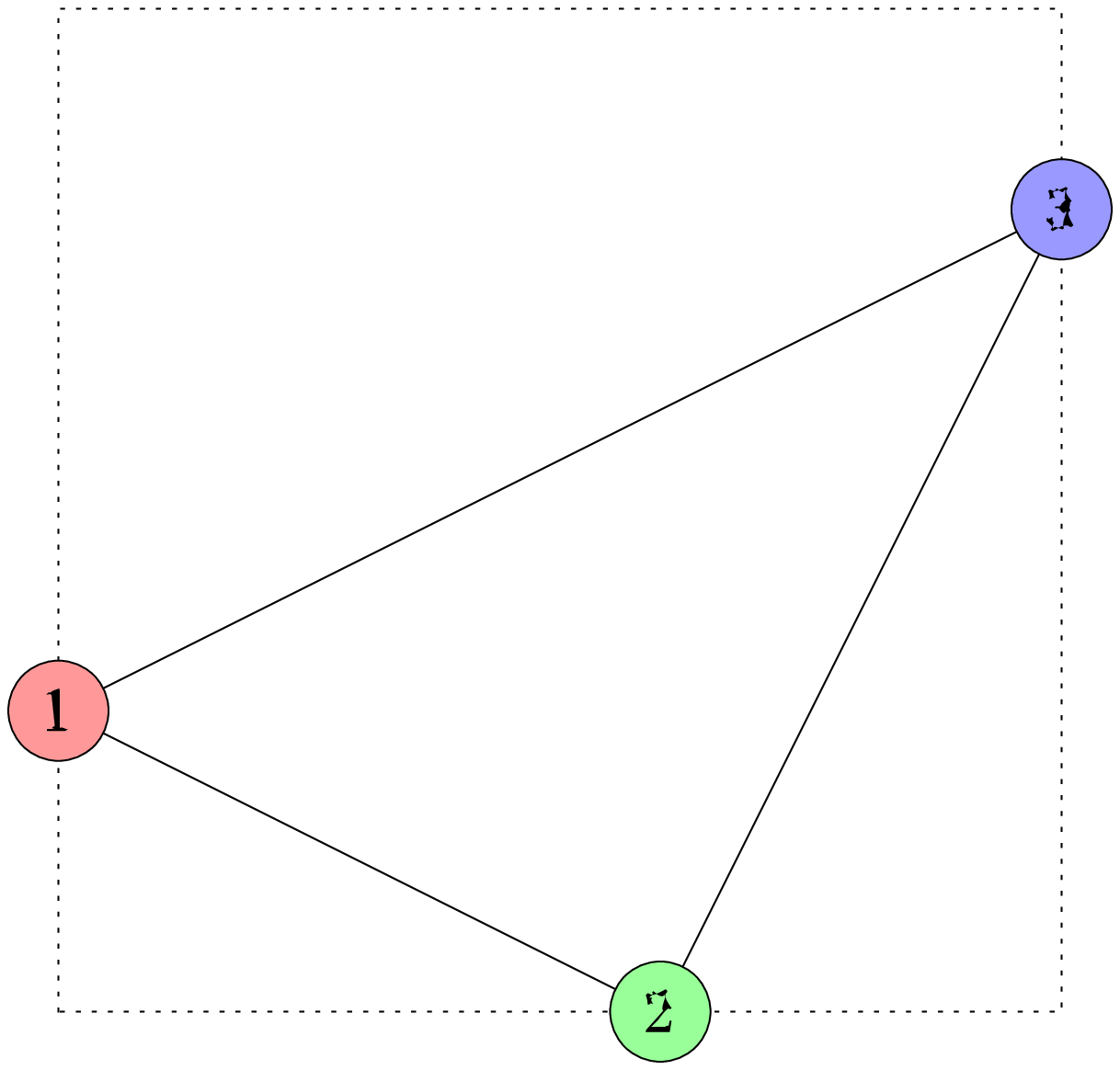} \quad \includegraphics[height=1.7in]{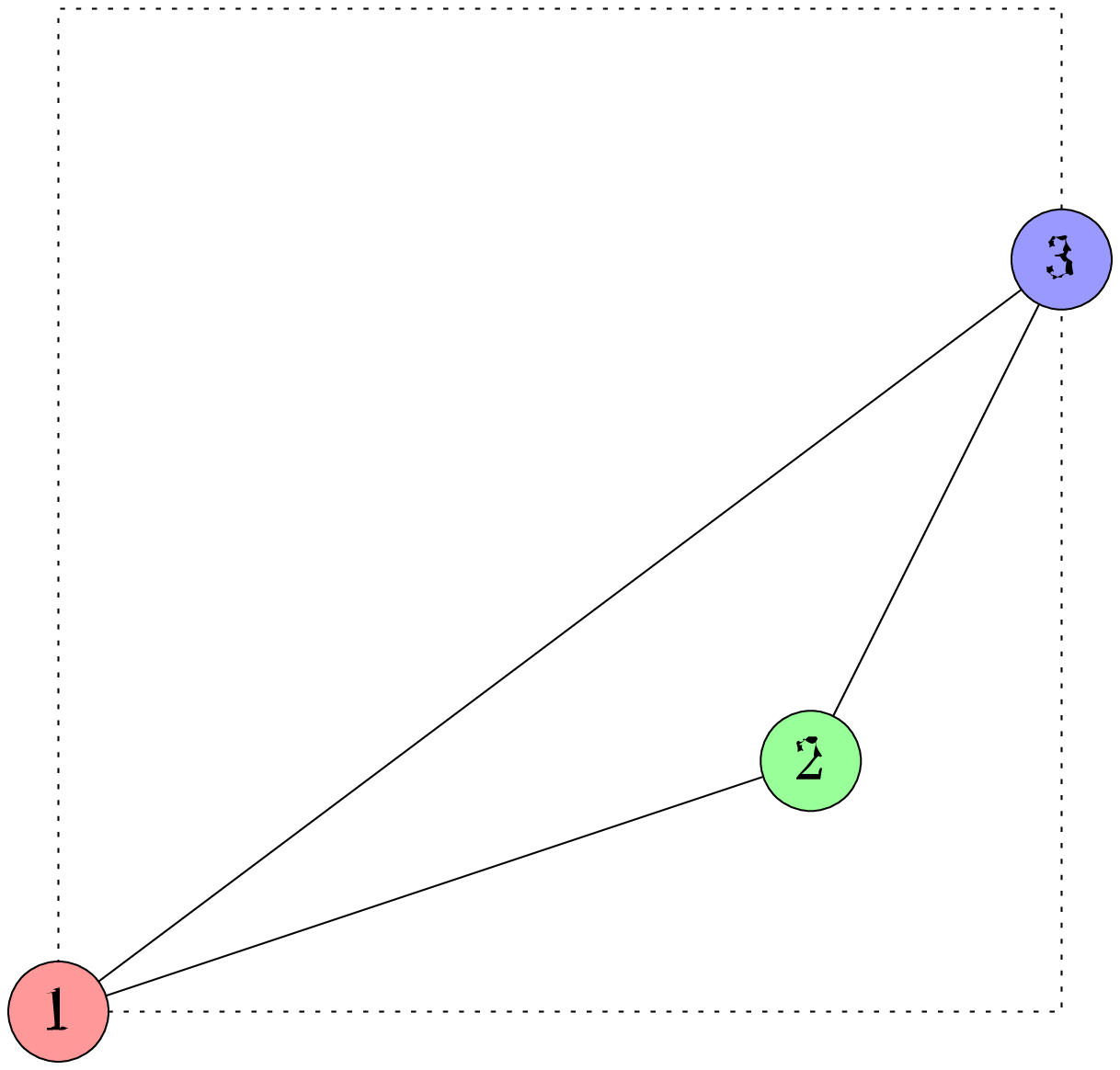}
\caption{Two integral configurations of pieces illustrating Proposition~\ref{P:triangular} when (a) $d_1/c_1 < 0$ and (b) $d_1/c_1 > 0$. (a) For $\cH^{-1/2}_{12}$, $\cH^{1/2}_{13}$, $\cH^{2/1}_{23}$, $x_1=0$, $y_2=0$, and $x_3=10$, the coordinates are $(0,3)$, $(6,0)$, and $(10,8)$, so $N=10$.  (b) For $\cH^{1/3}_{12}$, $\cH^{3/4}_{13}$, $\cH^{2/1}_{23}$, $x_1=0$, $y_2=0$, and $x_3=4$, we have $N=4$ because the coordinates are $(0,0)$, $(3,1)$, and $(4,3)$, which are not all on the boundary.}
\label{fig:threemoveNEW}
\end{figure}


\subsection{The golden parallelogram}\label{sec:config3p}\

Now we explore vertex configurations that use three moves.  We prepare for the general case by studying the \emph{semiqueen} $\pQ^{21}$, which has a horizontal, vertical, and diagonal move; we take its move set to be $\M = \{(1,0), (0,1), (-1,1)\}$.  

A {\em golden rectangle} is a rectangle whose sides are in the ratio $1{:}\phi$, $\phi$ being the golden ratio $\frac{1+\sqrt{5}}{2}$.\label{d:phi}  The rectangle that has side lengths $F_i$ and $F_{i+1}$, where the $F_i$\label{d:Fib} are Fibonacci numbers, is a close approximation to such a rectangle.  
(We index the Fibonacci numbers so that $F_0=F_1=1$.)

Many vertex configurations of $q$ semiqueens have denominator $F_{\lfloor q/2\rfloor}$.  One of them is the {\em golden rectangle configuration}, defined by the move hyperplanes  
\begin{gather*}
\cX_{4i,4i+1},\quad \cX_{4i+2,4i+6},\quad \cX_{4i+1,4i+3}, \\
\cY_{14},\quad \cY_{4i,4i+4},\quad \cY_{4i+2,4i+3},\quad \cY_{4i+3,4i+5},\\ 
\cH^{-1/1}_{2i+1,2i+2},
\end{gather*}
for all $i$ such that both indices fall between $1$ and $q$, inclusive, and the fixations $y_1=0$, $x_2=0$, and either $x_q=F_{\lfloor q/2\rfloor}$ if ${\lfloor q/2\rfloor}$ is even or $y_q=F_{\lfloor q/2\rfloor}$ if ${\lfloor q/2\rfloor}$ is odd.  
These fixations define the smallest square box that contains all pieces in the configuration.  They also serve to locate the configuration in the unit-square board, by giving the unique positive integer $N$ such that dividing by $N$ fits the shrunken configuration $\bz$ into the square board with three queens fixed on its boundary; thus $\bz$ is a vertex with denominator $\Delta(\bz)=N$  (see Lemma~\ref{L:DeltaN}).

Figure~\ref{fig:q21fib1}(a) shows the golden rectangle configuration of $12$ semiqueens; it fits in an $8\times 13$ rectangle.  Figure~\ref{fig:q21fib1}(b) is a configuration that has the same denominator and is similarly related to a discrete Fibonacci spiral (which will be explained in Section~\ref{sec:4move}, where it figures more prominently).

\begin{figure}[htbp]
\includegraphics[height=2.5in]{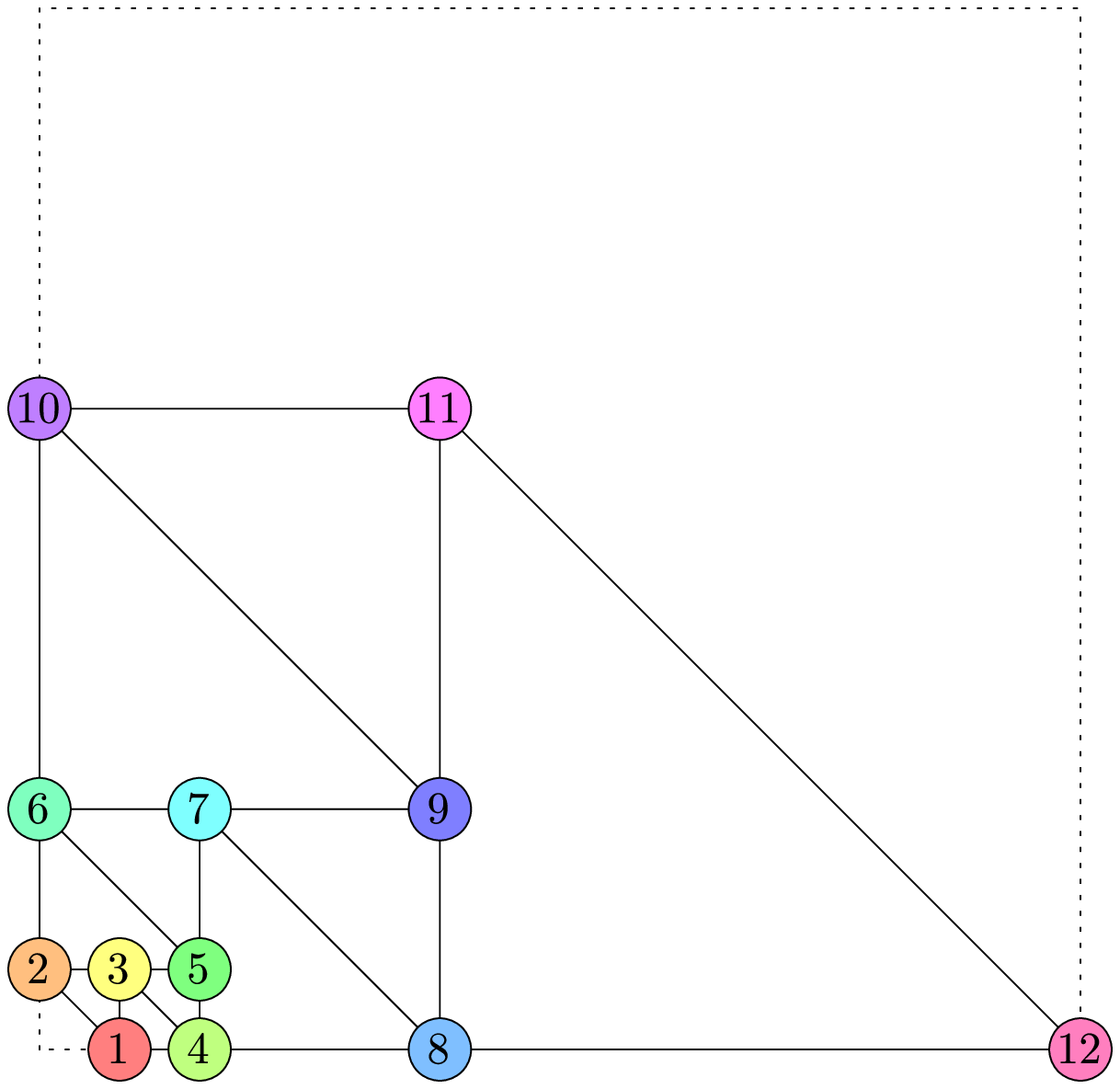}\qquad
\includegraphics[height=2.5in]{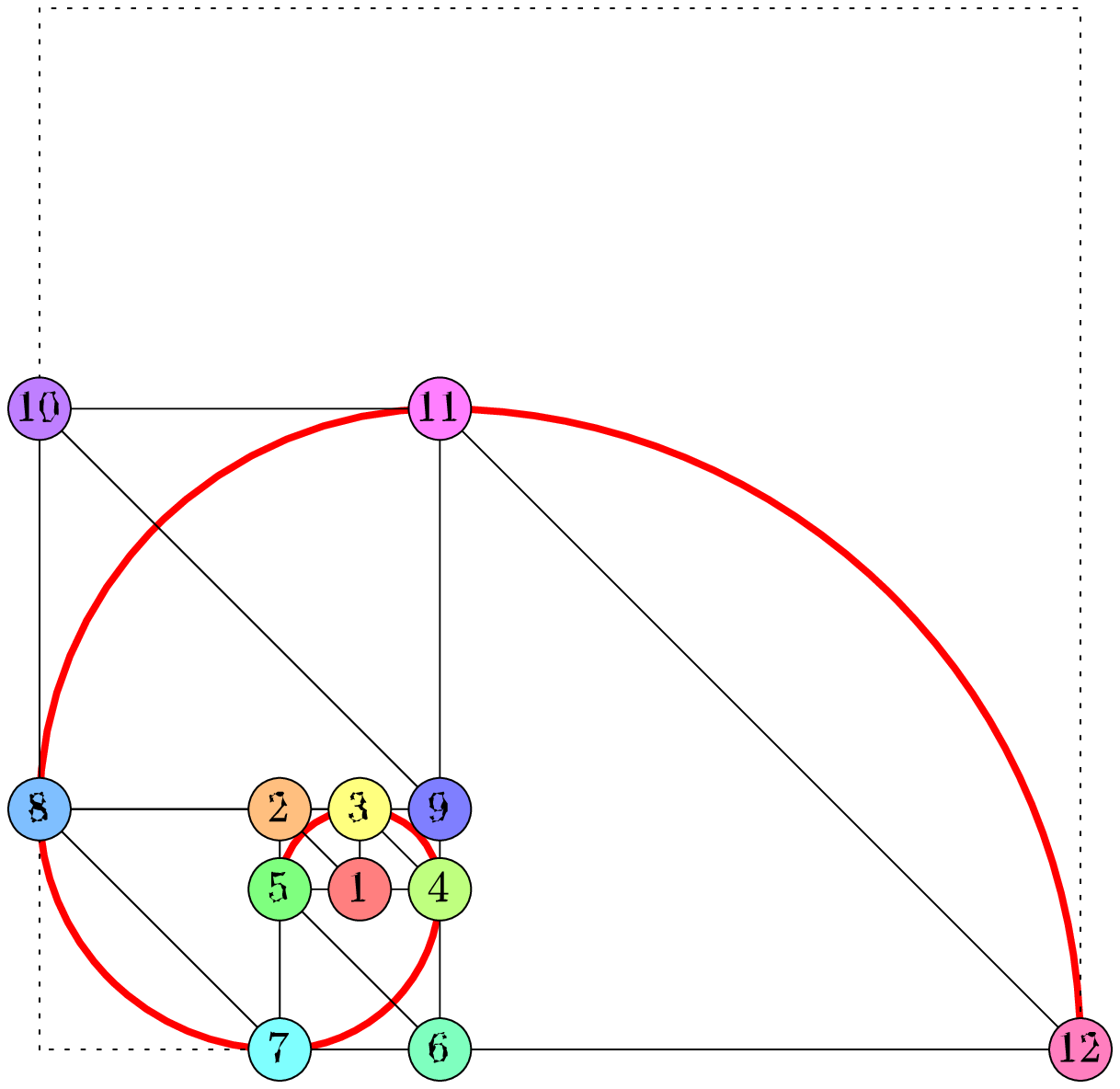}
\caption{(a) The golden rectangle configuration.  (b) A configuration based on a discrete Fibonacci spiral.}
\label{fig:q21fib1}
\end{figure}

It is straightforward to find the coordinates of $\pP_i$ in the golden rectangle configuration, which we present without proof.  We assume coordinates with origin in the lower left corner of Figure \ref{fig:q21fib1}(a).

\begin{prop}
For the semiqueen $\pP = \pQ^{21}$, when the pieces are arranged in the golden rectangle configuration, $\pP_1$ is in position $(1,0)$ and, for $i\geq 2$, $\pP_{i}$ is in position
\begin{align*}
(F_{\lfloor i/2\rfloor},0) &\qquad\text{if\/ }  i\equiv 0 \bmod 4, \\
(F_{\lfloor i/2\rfloor},F_{\lfloor i/2\rfloor-1}) &\qquad\text{if\/ } i\equiv 1 \bmod 4, \\
(0,F_{\lfloor i/2\rfloor}) &\qquad\text{if\/ } i\equiv 2 \bmod 4, \\
(F_{\lfloor i/2\rfloor-1},F_{\lfloor i/2\rfloor}) &\qquad\text{if\/ } i\equiv 3 \bmod 4.
\end{align*}
The step from $\pP_{i-1}$ to $\pP_i$ is 
\begin{align*}
F_{\lfloor i/2\rfloor-1}(1,-1) &\qquad\text{if\/ }  i\equiv 0 \bmod 4, \\
F_{\lfloor i/2\rfloor-1}(0,1) &\qquad\text{if\/ } i\equiv 1 \bmod 4, \\
F_{\lfloor i/2\rfloor-1}(-1,1) &\qquad\text{if\/ } i\equiv 2 \bmod 4, \\
F_{\lfloor i/2\rfloor-1}(1,0) &\qquad\text{if\/ } i\equiv 3 \bmod 4.
\end{align*}
\end{prop}

The next key idea is that we can apply a linear transformation to the golden rectangle configuration to create {\bf six} \emph{golden parallelogram configurations} (some of which may coincide if there is symmetry in the move set) for any piece with three (or more) moves.  
To define the golden parallelogram, in the golden rectangle configuration consider the semiqueens $\pQ^{21}_1$ at position $(1,0)$, $\pQ^{21}_2$ at $(0,1)$, and $\pQ^{21}_3$ at $(1,1)$.  They form the smallest possible triangle. 
For an arbitrary piece $\pP$ with moves $m_1$, $m_2$, and $m_3$, we consider the smallest integral triangle involving three copies of $\pP$, which we discussed in Proposition~\ref{P:triangular}.  We apply to the golden rectangle configuration a linear transformation that takes vectors $( 1,0)$ and $( 0,1)$ to any two of the vectors $w_1m_1$, $w_2m_2$, and $w_3m_3$, with a minus sign on one of them if needed to ensure that the third side of the triangle has the correct orientation.  That transforms the golden rectangle with the $\pQ^{21}_i$ in their locations to a golden parallelogram with pieces $\pP_i$ in the transformed locations and with $\pP_1, \pP_2, \pP_3$ forming the aforementioned smallest triangle; hence, there are six possible golden parallelograms.

\begin{exam}\label{ex:parallelograms}
For the three-move partial nightrider (Example \ref{ex:threemove}) the vectors are $w_1m_1=( 6,-3)$, $w_2m_2=(-10,-5)$, and $w_3m_3=( 4,8)$.  The corresponding six distinct golden parallelogram configurations are in Figure~\ref{fig:parallelograms}.  The precise linear transformations are given in Table~\ref{tab:parallelograms}.  Of these six parallelograms, the one yielding the largest denominator is that in the upper left. 
\end{exam}
\begin{figure}[htbp]
\includegraphics[height=1.9in]{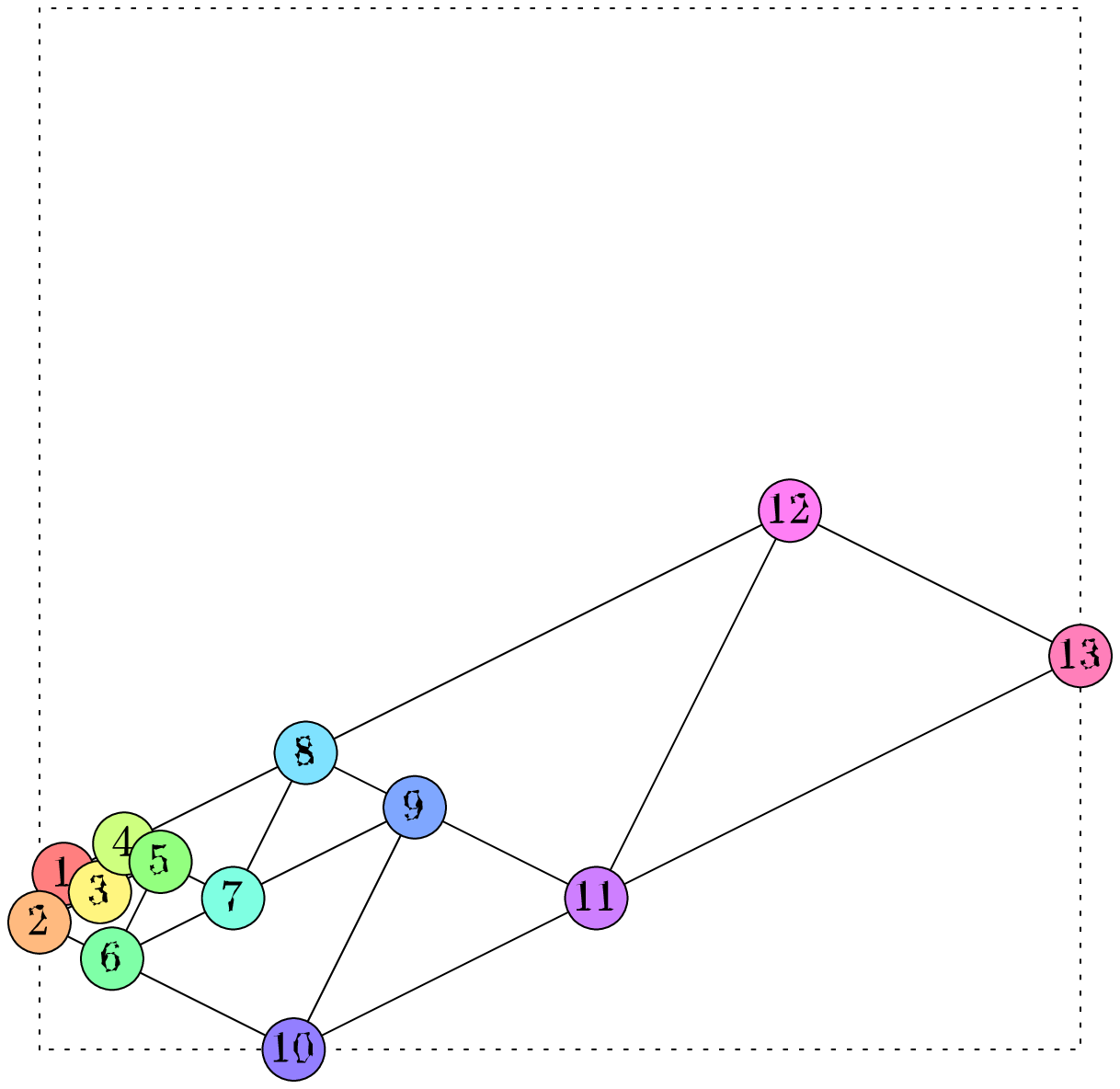}
\includegraphics[height=1.9in]{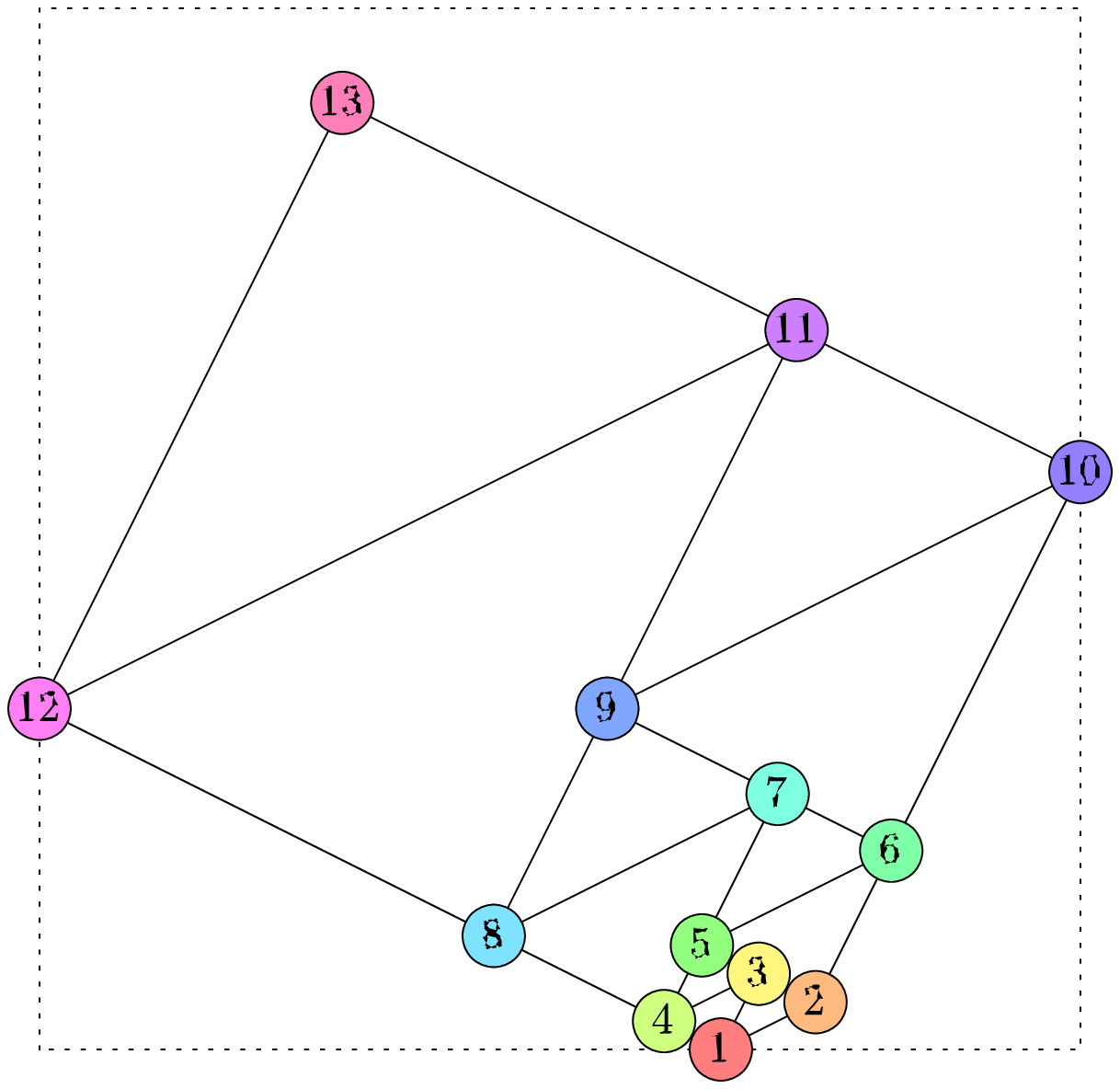}
\includegraphics[height=1.9in]{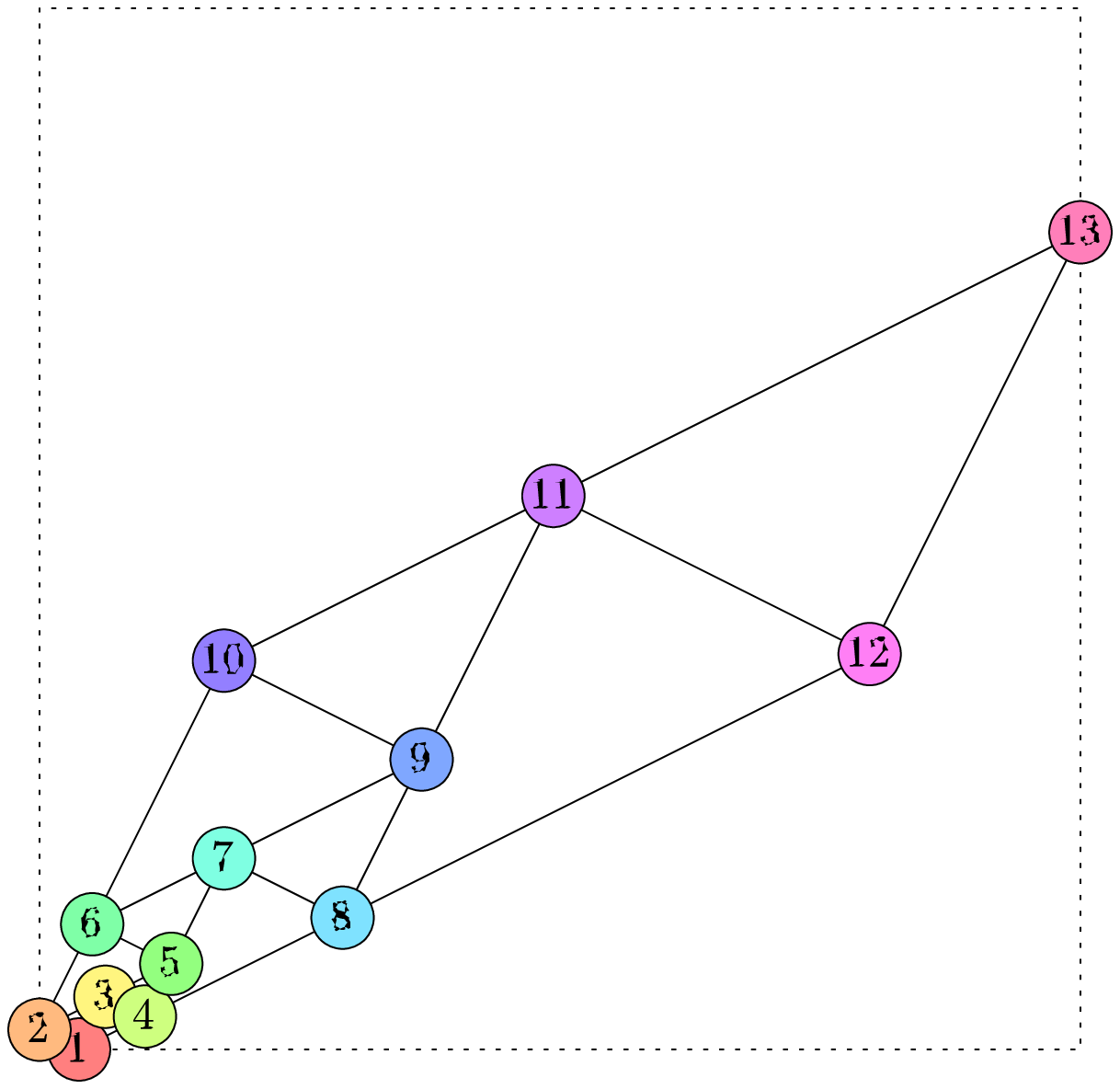}
\includegraphics[height=1.9in]{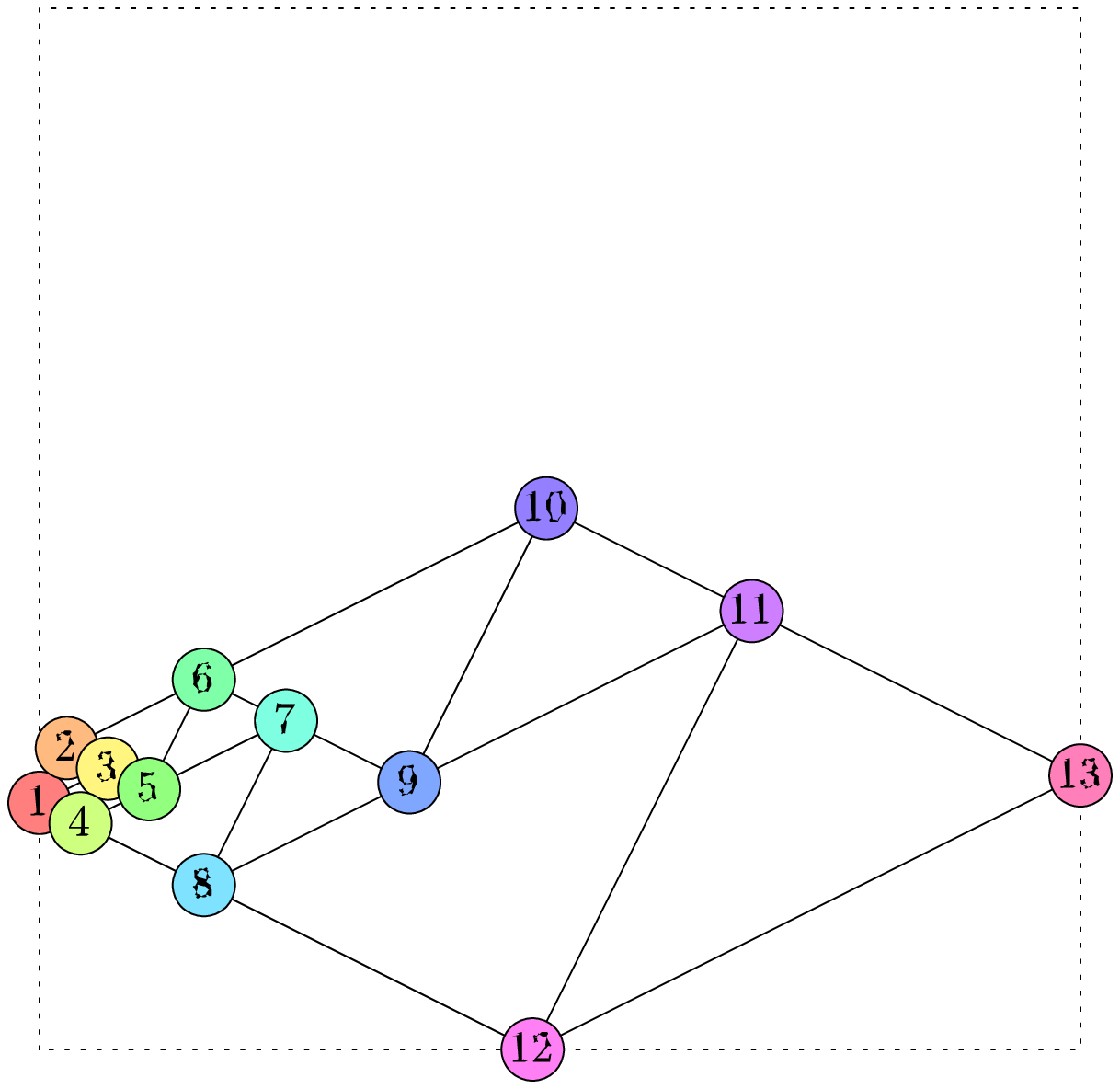}
\includegraphics[height=1.9in]{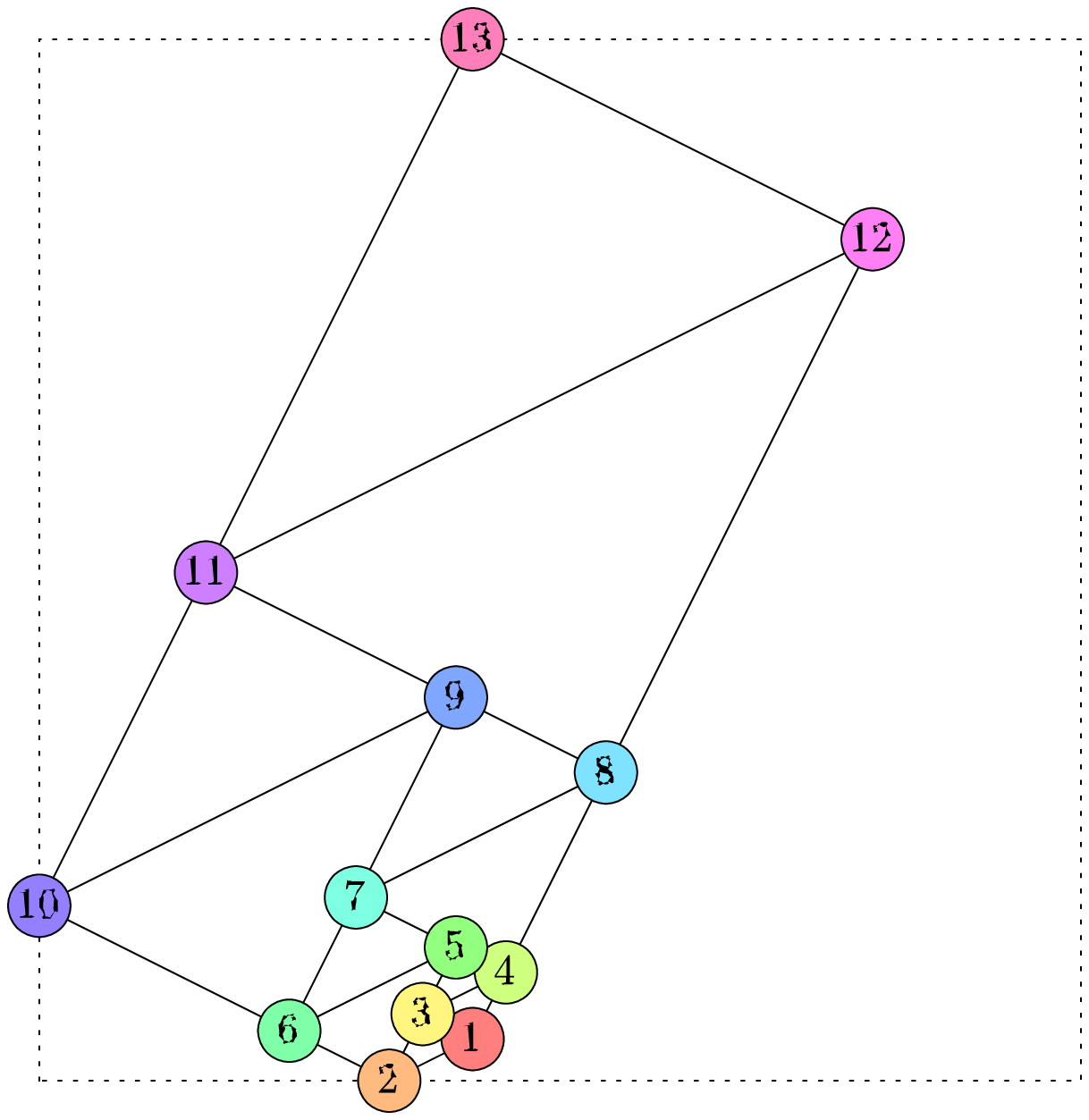}
\includegraphics[height=1.9in]{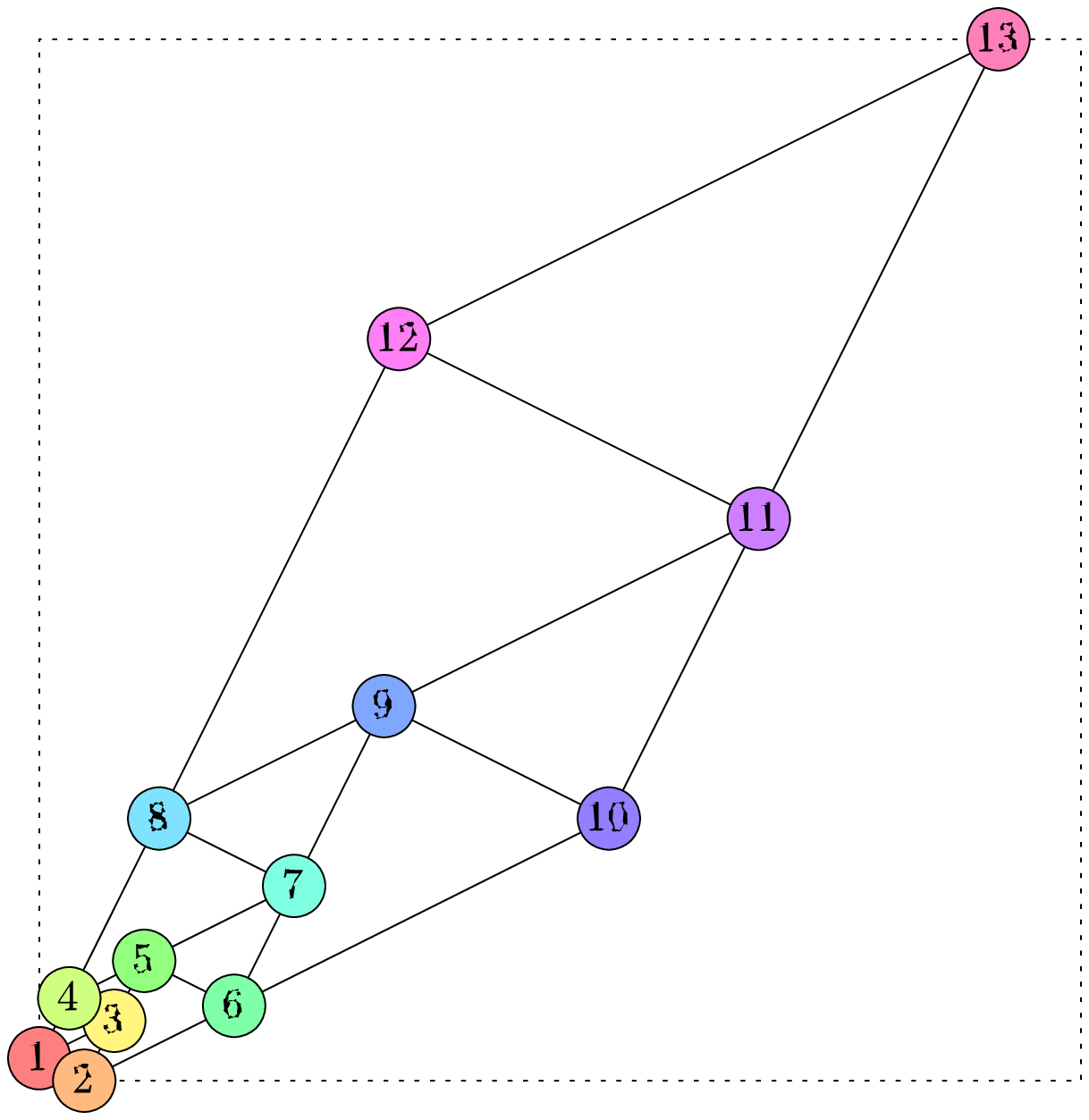}
\caption{The six golden parallelograms for 13 three-move partial nightriders.  The corresponding linear transformations are given in Table~\ref{tab:parallelograms}.}
\label{fig:parallelograms}
\end{figure}

\begin{table}[htbp]
\begin{tabular}{|c|c|c|c|}\hline
\begin{tabular}{c}
Transformation
\end{tabular} & 
\begin{tabular}{c}
$( 1,0) \mapsto ( 10,5)$ \\  
$( 0,1) \mapsto ( 6,-3)$ 
\end{tabular} \vstrut{20pt}&
\begin{tabular}{c}
$( 1,0) \mapsto ( -6,3)$ \\  
$( 0,1) \mapsto ( 4,8)$ 
\end{tabular} &
\begin{tabular}{c}
$( 1,0) \mapsto ( 10,5)$ \\  
$( 0,1) \mapsto ( 4,8)$ 
\end{tabular} \\ 
\hfill $\Delta$ \hspace*{1em}&\hfill 172 \hspace*{1em}&\hfill 110 \hspace*{1em}&\hfill  158 \hspace*{1em}\\ \hline
\begin{tabular}{c}
Transformation
\end{tabular} & 
\begin{tabular}{c}
$( 1,0) \mapsto ( 6,-3)$ \\  
$( 0,1) \mapsto ( 10,5)$ 
\end{tabular} \vstrut{20pt}&
\begin{tabular}{c}
$( 1,0) \mapsto ( 4,8)$ \\  
$( 0,1) \mapsto ( -6,3)$ 
\end{tabular} &
\begin{tabular}{c}
$( 1,0) \mapsto ( 4,8)$ \\  
$( 0,1) \mapsto ( 10,5)$ 
\end{tabular} \\ 
\hfill $\Delta$ \hspace*{1em}&\hfill  152 \hspace*{1em}&\hfill 125 \hspace*{1em}&\hfill  139 \hspace*{1em}\\ \hline
\end{tabular}
\smallskip

\label{tab:parallelograms}
\caption{The linear transformations corresponding to the golden parallelogram configurations of 13 pieces in Figure~\ref{fig:parallelograms}, along with the denominator $\Delta$ for each configuration.}
\end{table}

These golden parallelograms appear to maximize the denominator; from them we may infer conjectural formulas for the largest denominators.

\begin{conj}\label{Cj:3movemax}
For a piece with exactly three moves, one of the golden parallelogram configurations gives a vertex with the largest denominator.  
\end{conj}

Suppose the linear transformation that creates a golden parallelogram carries $(1,0) \mapsto w_1m_1 = (w_1c_1,w_1d_1)$ and $(0,1) \mapsto w_2m_2 = (w_2c_2,w_2d_2)$.  It is possible to write an explicit formula for the denominator of the resulting golden parallelogram configuration.  The computation has not more than $128 = 4\cdot2^4\cdot2$ cases, with one case for each value of $q \mod 4$ and one subcase for each of the $2^4$ sign patterns of the components of $w_1m_1$ and $w_2m_2$ (sign $0$ can be combined with sign $+$), and in some of those subcases one further subcase for each of the $2^2$ magnitude relations between $|w_1c_1|$ and $|w_2c_2|$ or between $|w_1d_1|$ and $|w_2d_2|$.  We have not computed the entire formula.  
Nevertheless, we can give an exponential lower bound, thereby obtaining a lower bound on the denominator of the inside-out polytope.

\begin{thm}\label{T:exponential}
The denominators $D_q(\pP)$ of any piece that has three or more moves increase at least exponentially in $q$.  Specifically, they are bounded below by $\frac12\phi^{q/2}$ when $q\geq12$, where $\phi$ is the golden ratio.
\end{thm}

\begin{proof}
To prove the theorem it suffices to produce a vertex of $(\cube, \cA_\pP^q)$ with denominator exceeding $\phi^{q/2}$.  

First consider the semiqueen.  The points $\pQ^{21}_1$ and $\pQ^{21}_{4j}$ of the golden rectangle have coordinates $(1,0)$ and $(F_{2j},0)$.  Letting $q=4j$ or $4j+1$ gives an $x$-difference of $F_{2j}-1$ for a golden rectangle of $q$ pieces.  Similarly, letting $q=4j+2$ or $4j+3$ gives a $y$-difference of $F_{2j+1}$.  The golden rectangle is a vertex configuration so it follows by Lemma~\ref{L:DeltaN} that the vertex $\bz$ has $\Delta(\bz)\geq F_{\lfloor q/2 \rfloor}-1$.  A calculation shows that $F_{\lfloor j \rfloor}-1 > \frac12 \phi^{j+\frac12}$ for $j\geq6$.  The theorem for $\pQ^{21}$ follows.

An arbitrary piece with three (or more) moves has a golden parallelogram configuration formed from the golden rectangle by the linear transformation $(1,0)\mapsto w_1m_1$ and $(0,1)\mapsto w_2m_2$.  We may choose these moves from at least three, so we can select $m_1$ to have $c_1\neq0$ and $m_2$ to have $d_2\neq0$.  The displacement from $\pQ^{21}_1$ to $\pQ^{21}_{4j}$ becomes that from $\pP_1$ at $w_1m_1$ to $\pP_{4j}$ at $F_{2j}w_1m_1$.  This displacement is $(F_{2j}-1)(w_1c_1,w_1d_1)$.  Since $c_1 \neq 0$, the $x$-displacement is at least that for $\pQ^{21}$; therefore the denominator of the corresponding vertex for $\pP$ is bounded below by $F_{2j}-1$, just as it is for $\pQ^{21}$.  Similarly, the $y$-displacement for $q=4j+2$ is bounded below by $F_{2j+1}$.  This reduces the problem to the semiqueen, which is solved.
\end{proof}

We know from Proposition~\ref{P:Dincrease} that $D_q(\pP)$ is weakly increasing.  If, as we believe, the period equals $D_q(\pP)$, then the period increases at least exponentially for any piece with more than two moves.  

We think any board has a similar lower bound, say $C(\cB)\phi^{q/2}$ where $C(\cB)$ is a constant depending upon $\cB$, but we ran into technical difficulties trying to prove it.

\begin{exam}\label{ex:Q21parallelogram}
The semiqueen has (up to symmetry) only one other golden parallelogram besides the golden rectangle; it is shown in Figure~\ref{fig:q21fib2}.  It has a larger denominator than the golden rectangle configuration when $q$ is odd and $q\geq 7$. 
\begin{figure}[htbp]
\includegraphics[height=2.5in]{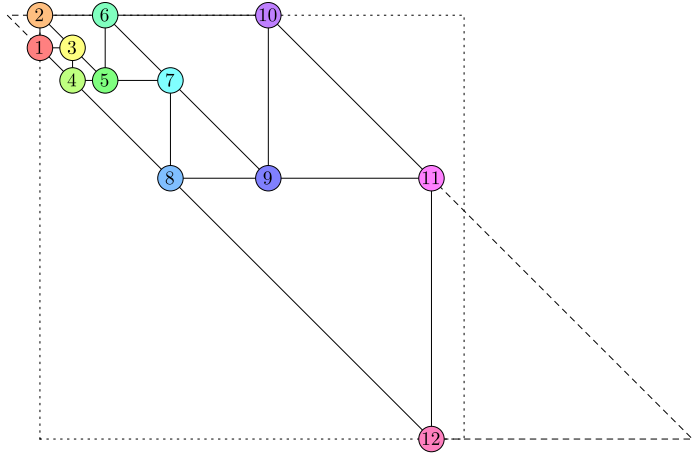}
\caption{A golden parallelogram configuration of $q$ semiqueens has the largest denominator when $q$ is odd.}
\label{fig:q21fib2}
\end{figure}
If we put $\pP_2$ in position $(0,0)$, then $\pP_1$ is in position $(0,-1)$ and for $i\geq 2$, $\pP_{i}$ is in position
\begin{align*}
(F_{\lfloor i/2\rfloor}-1,-F_{\lfloor i/2\rfloor}) &\qquad\text{if\/ } i\equiv 0 \bmod 4, \\
(F_{\lfloor i/2\rfloor+1}-1,-F_{\lfloor i/2\rfloor}) &\qquad\text{if\/ } i\equiv 1 \bmod 4, \\
(F_{\lfloor i/2\rfloor}-1,0) &\qquad\text{if\/ } i\equiv 2 \bmod 4, \\
(F_{\lfloor i/2\rfloor+1}-1,-F_{\lfloor i/2\rfloor-1}) &\qquad\text{if\/ } i\equiv 3 \bmod 4.
\end{align*}
Conjecture \ref{Cj:3movemax} specializes to:

\begin{conj}\label{conj:semiqueen}
The largest denominator of a vertex for $q$ semiqueens~$\pQ^{21}$ is $F_{\lfloor q/2\rfloor}$ if $q$ is even and $F_{\lfloor q/2\rfloor+1}-1$ if $q$ is odd.
\end{conj}
\end{exam}

\begin{exam}\label{ex:Q12parallelograms}
The \emph{trident} $\pQ^{12}$ has move set $\M = \{(0,1), (1,1), (-1,1)\}$.  It gives the three distinct golden parallelogram configurations shown in Figure~\ref{fig:q12fib1}. 
\begin{figure}[htbp]
\includegraphics[height=2in]{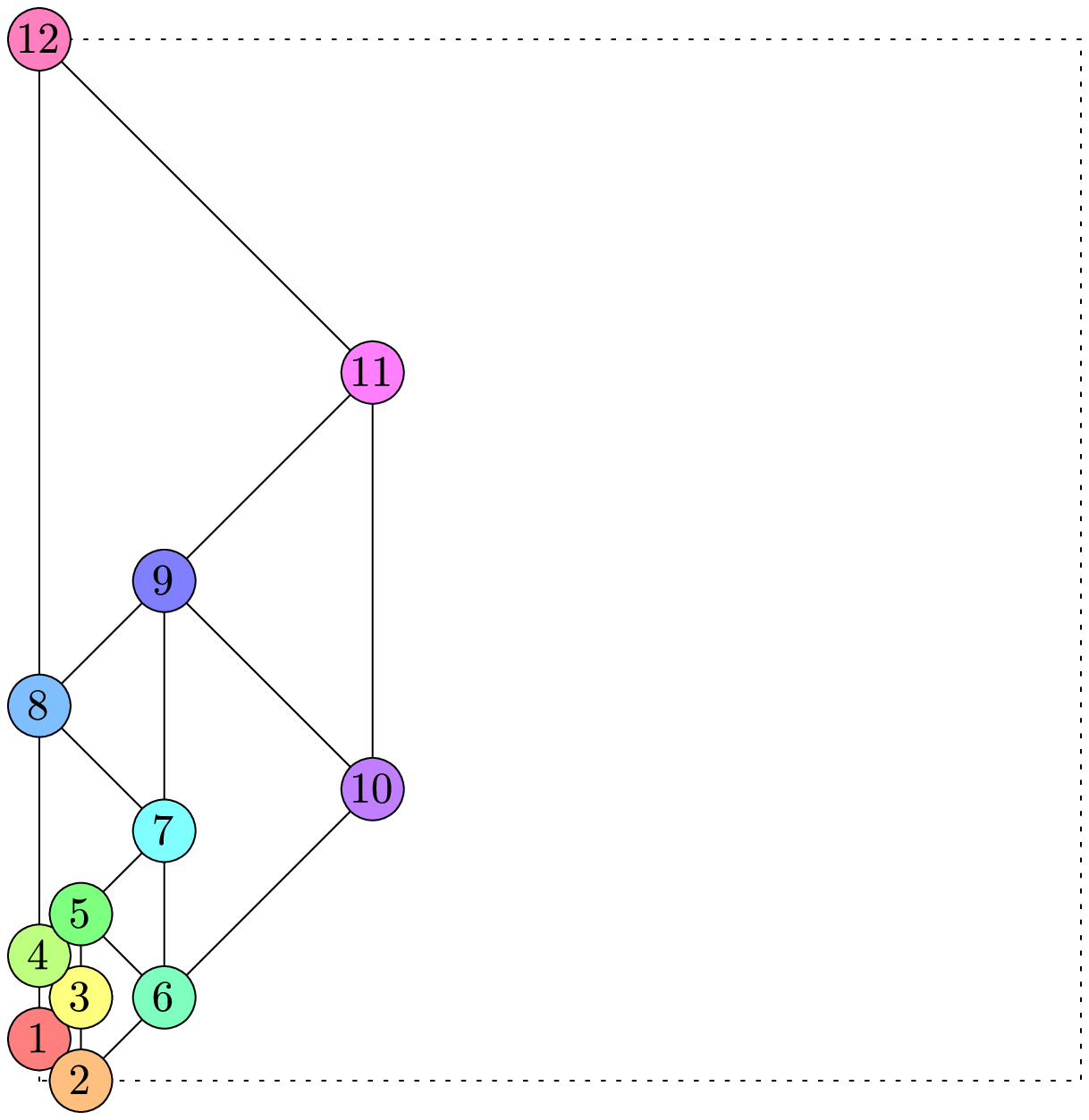}
\includegraphics[height=2in]{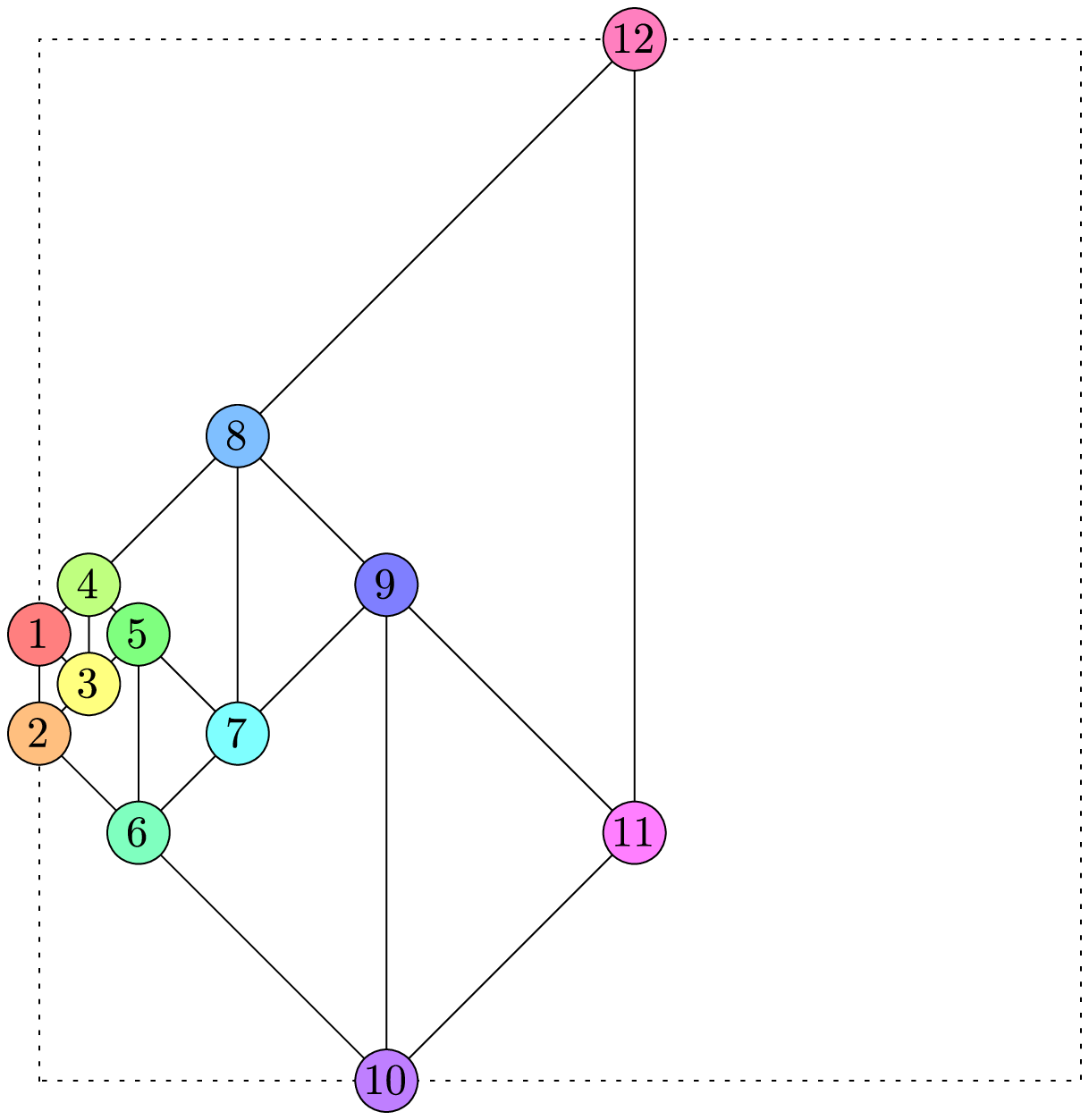}
\includegraphics[height=2in]{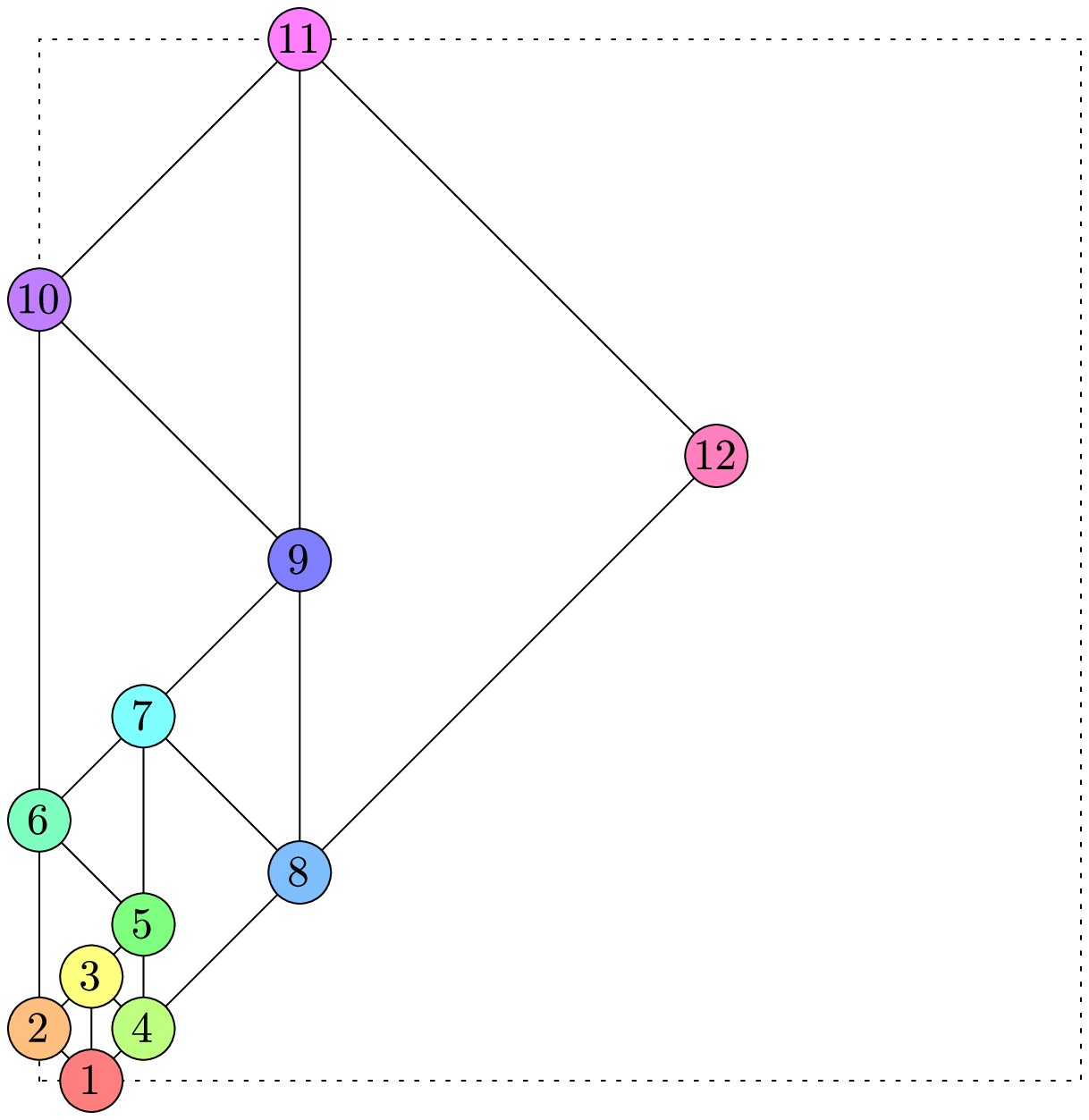}
\caption{The three golden parallelograms for the trident $\pQ^{12}$. For twelve pieces, the denominators are 25, 21, and 20, respectively.}
\label{fig:q12fib1}
\end{figure}
The piece positions for $i\geq 2$ for the configuration in Figure~\ref{fig:q12fib1}(a) are
\begin{align*}
(0,2F_{\lfloor i/2\rfloor}-1) &\qquad\text{if\/ } i\equiv 0 \bmod 4, \\
(F_{\lfloor i/2\rfloor-1},F_{\lfloor i/2\rfloor+2}-1) &\qquad\text{if\/ } i\equiv 1 \bmod 4, \\
(F_{\lfloor i/2\rfloor},F_{\lfloor i/2\rfloor}-1) &\qquad\text{if\/ } i\equiv 2 \bmod 4, \\
(F_{\lfloor i/2\rfloor},F_{\lfloor i/2\rfloor+1}+F_{\lfloor i/2\rfloor-1}-1) &\qquad\text{if\/ } i\equiv 3 \bmod 4, 
\end{align*}
so the largest denominator for such a configuration is 
\begin{align*}
2F_{\lfloor q/2\rfloor}-1 &\qquad\text{if }  q\equiv 0 \bmod 4, \\
F_{\lfloor q/2\rfloor+2}-1 &\qquad\text{if }  q\equiv 1 \bmod 4, \\
F_{\lfloor q/2\rfloor+1}-1 &\qquad\text{if }  q\equiv 2 \bmod 4, \\
F_{\lfloor q/2\rfloor+1}+F_{\lfloor q/2\rfloor-1}-1 &\qquad\text{if }  q\equiv 3 \bmod 4. 
\end{align*}
On the other hand, in Figure~\ref{fig:q12fib1}(c), the piece positions for $i\geq 2$ are
\begin{align*}
(F_{\lfloor i/2\rfloor},F_{\lfloor i/2\rfloor}-1) &\qquad\text{if\/ } i\equiv 0 \bmod 4, \\
(F_{\lfloor i/2\rfloor},F_{\lfloor i/2\rfloor+1}+F_{\lfloor i/2\rfloor-1}-1) &\qquad\text{if\/ } i\equiv 1 \bmod 4, \\
(0,2F_{\lfloor i/2\rfloor}-1) &\qquad\text{if\/ } i\equiv 2 \bmod 4, \\
(F_{\lfloor i/2\rfloor-1},F_{\lfloor i/2\rfloor+2}-1) &\qquad\text{if\/ } i\equiv 3 \bmod 4, 
\end{align*}
which yields a largest denominator of such a configuration of 
\begin{align*}
F_{\lfloor q/2\rfloor+1}-1 &\qquad\text{if }  q\equiv 0 \bmod 4, \\
F_{\lfloor q/2\rfloor+1}+F_{\lfloor q/2\rfloor-1}-1 &\qquad\text{if }  q\equiv 1 \bmod 4, \\
2F_{\lfloor q/2\rfloor}-1 &\qquad\text{if }  q\equiv 2 \bmod 4, \\
F_{\lfloor q/2\rfloor+2}-1 &\qquad\text{if }  q\equiv 3 \bmod 4. 
\end{align*}
Conjecture \ref{Cj:3movemax} specializes to:

\begin{conj} \label{conj:trident}
The largest denominator of a vertex for $q$ tridents~$\pQ^{12}$ is 
$2F_{q/2}-1$ if $q$ is even and $F_{(q+3)/2}-1$ if $q$ is odd.
\end{conj}

There is a remarkable symmetry between the piece positions in configurations (a) and (c).  The position formulas for $i \equiv r \mod 4$ in (c) are identical to those for $i \equiv r-2 \mod 4$ in (a).  We do not know why.
\end{exam}

\sectionpage
\section{Pieces with Four or More Moves}\label{sec:4move}

When a piece has four or more moves, the diversity of vertex configurations increases dramatically and the denominators grow much more quickly.  Again we start with the piece with the simplest four moves, the queen.  

The {\em Fibonacci spiral} is an approximation to the golden spiral (the logarithmic spiral with growth factor $\phi$); it is obtained by arranging in a spiral pattern squares of Fibonacci side length, each with a quarter circle inscribed, as shown in Figure~\ref{fig:FibSpiral}(a).

\begin{figure}[htbp]
	\raisebox{.066in}{\includegraphics[height=1.4in]{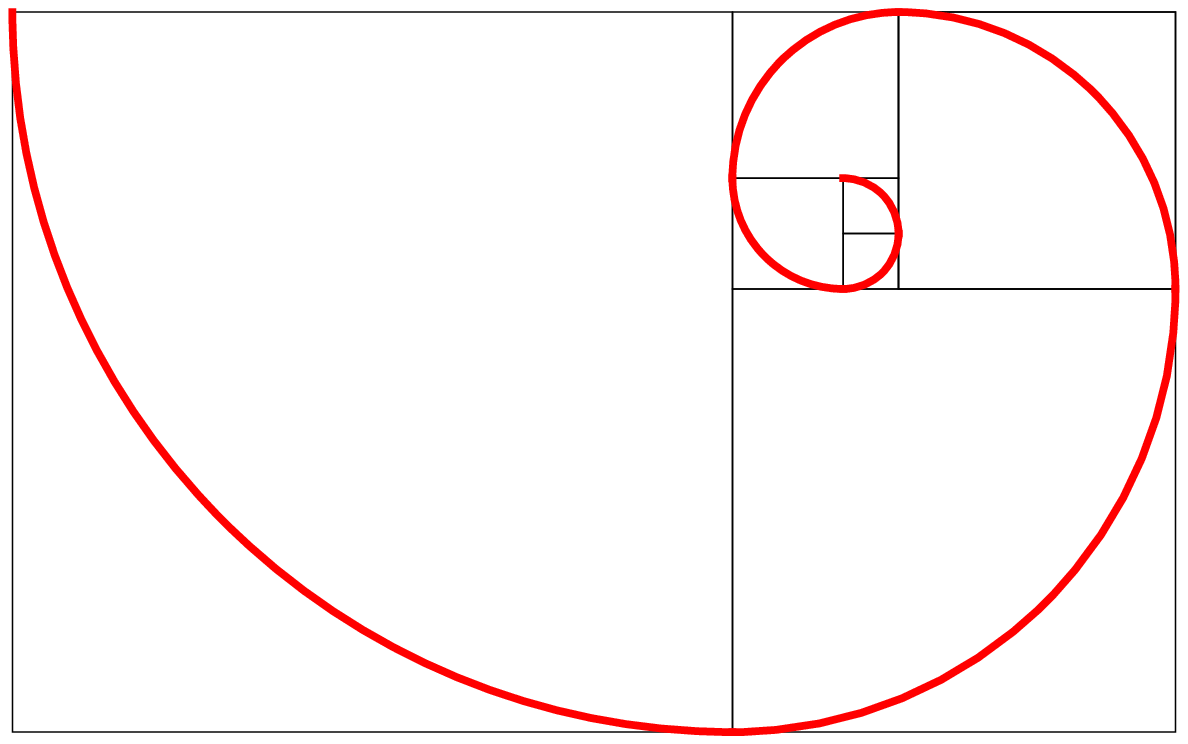}}
	\quad
		\includegraphics[height=2.25in]{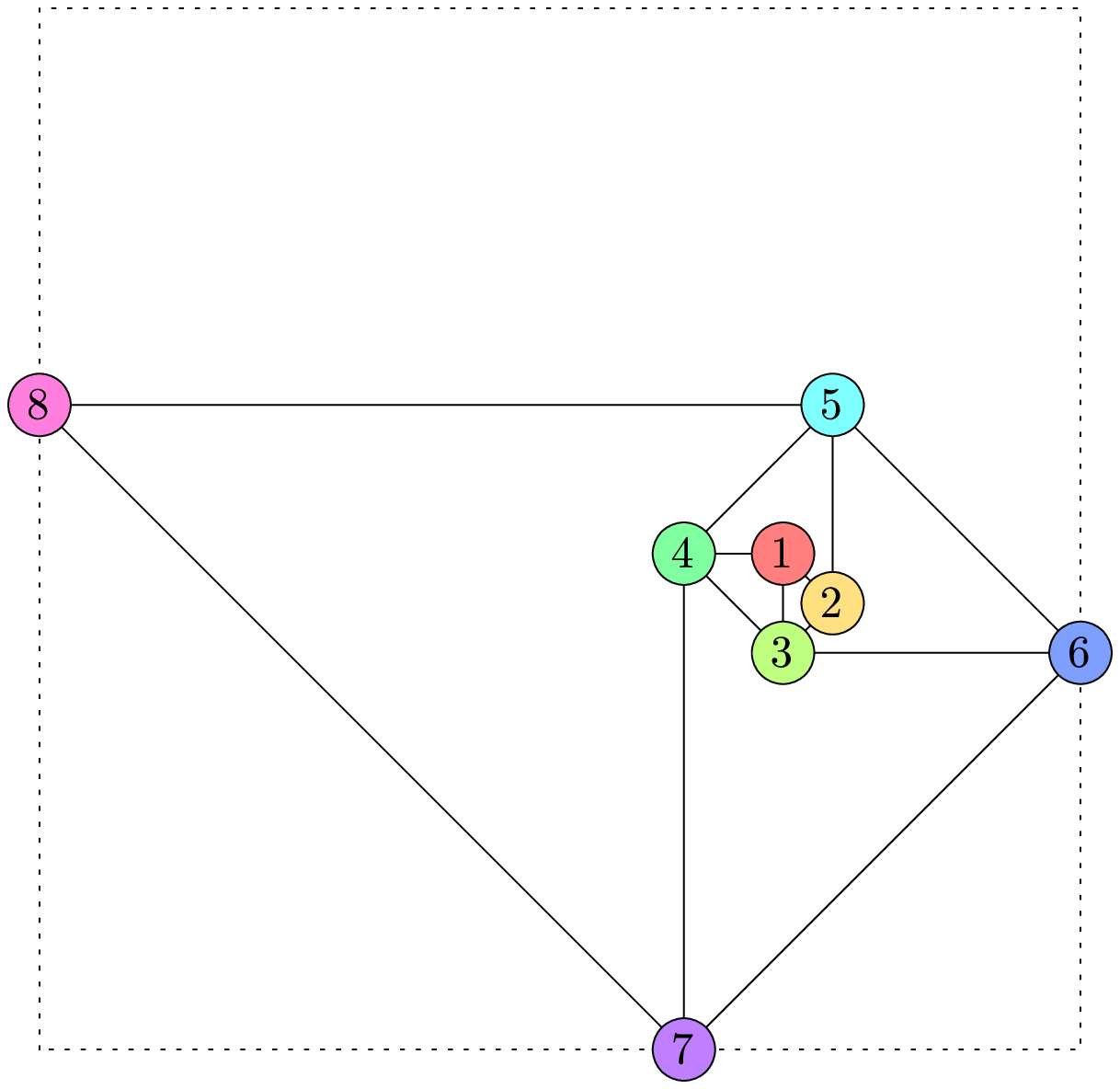}
\par	\quad (a) \hspace{2.1in} (b)
\caption{(a) The Fibonacci spiral. 	(b) The discrete Fibonacci spiral for eight queens.}
	\label{fig:FibSpiral}
\end{figure}

The {\em discrete Fibonacci spiral} with $q$ queens is defined by the move hyperplanes  
\[
\cH^{+1/1}_{2i,2i+1},\quad \cH^{-1/1}_{2i+1,2i+2},\quad \cX_{1,3},\quad  \cX_{2i,2i+3},\quad \cY_{2i+1,2i+4}
\] 
for all $i$ such that both indices fall between $1$ and $q$, inclusive, and fixations for pieces $\pP_{q-2}$, $\pP_{q-1}$, and $\pP_{q}$.  The fixations are
\begin{align*}
x_q=0,\ y_{q-1}=0,\ x_{q-2}=F_q &\qquad\text{if } q\equiv 0 \bmod 4, \\
x_q=0,\ y_{q}=0,\ x_{q-2}=F_q &\qquad\text{if } q\equiv 1 \bmod 4, \\
x_q=F_q,\ y_{q}=0,\ x_{q-2}=0 &\qquad\text{if } q\equiv 2 \bmod 4, \\
y_q=F_q,\ x_{q-1}=0,\ y_{q-2}=0 &\qquad\text{if } q\equiv 3 \bmod 4. \\
\end{align*}
Figure~\ref{fig:FibSpiral}(b) shows the discrete Fibonacci spiral of $8$ queens.

The bounding rectangle of the discrete Fibonacci spiral with $q$ queens has dimensions $F_q$ by $F_{q-1}$ so the vertex's denominator is $F_q$.

\begin{conj} 
The largest denominator that appears in any vertex configuration for $q$ queens is $F_q$.
\end{conj}

The queen appears to have an extremely special property, not shared with three-piece riders nor with other four-piece riders.  The initial data (for $q\leq 9$) seem to indicate that it is possible to construct vertex configurations that generate {\em all} denominators up to $F_q$.

\begin{conj}\label{conj:denomq}
For every integer $\delta$ between $1$ and $F_q$ inclusive, there exists a vertex configuration of $q$ queens with denominator $\delta$.
\end{conj}

\begin{exam}\label{ex:q8configs} 
The eighth Fibonacci number is $F_8=21$.  The spiral in Figure~\ref{fig:FibSpiral}(b) exhibits a denominator of $21$.  For each $\delta\leq F_7=13$ there is a vertex configuration of seven or fewer queens with denominator $\delta$ (we do not show them).  Figure~\ref{fig:q8configs} provides seven vertex configurations of eight queens in which the denominator ranges from 14 to 20, as one can tell from the size of the smallest enclosing square and Lemma~\ref{L:DeltaN}.
\begin{figure}[htbp]
\includegraphics[height=1.5in]{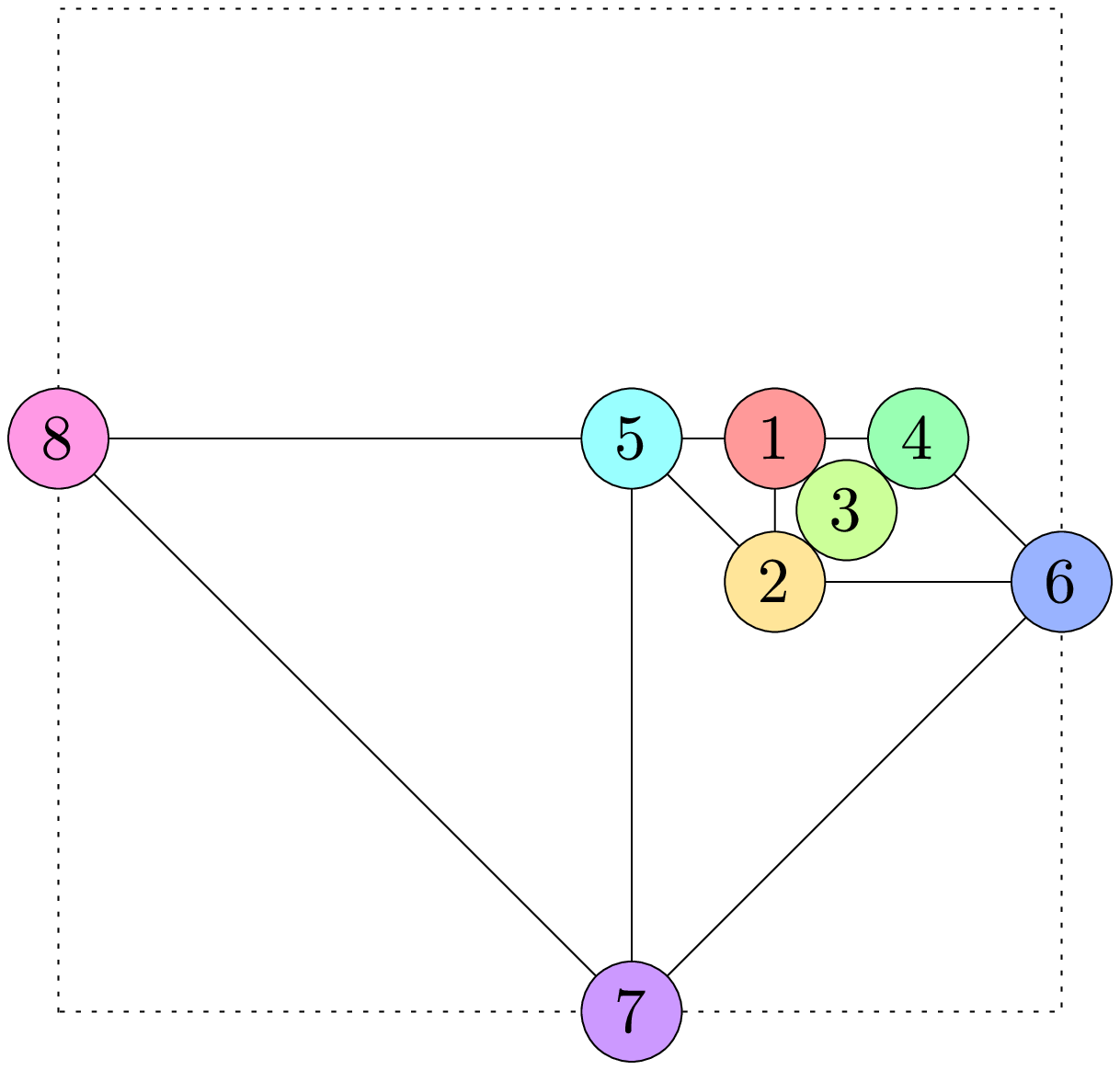}
\includegraphics[height=1.5in]{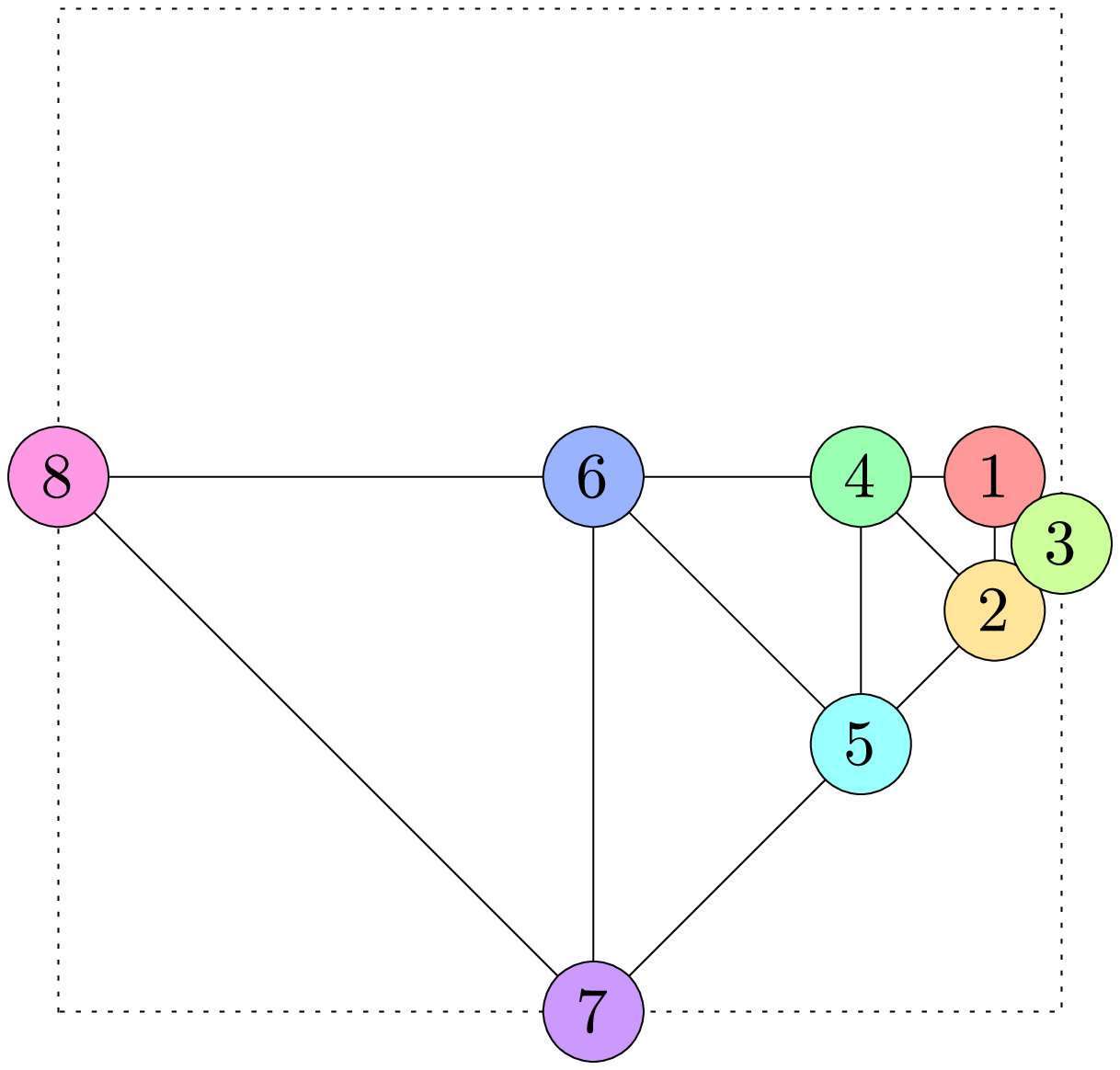}
\includegraphics[height=1.5in]{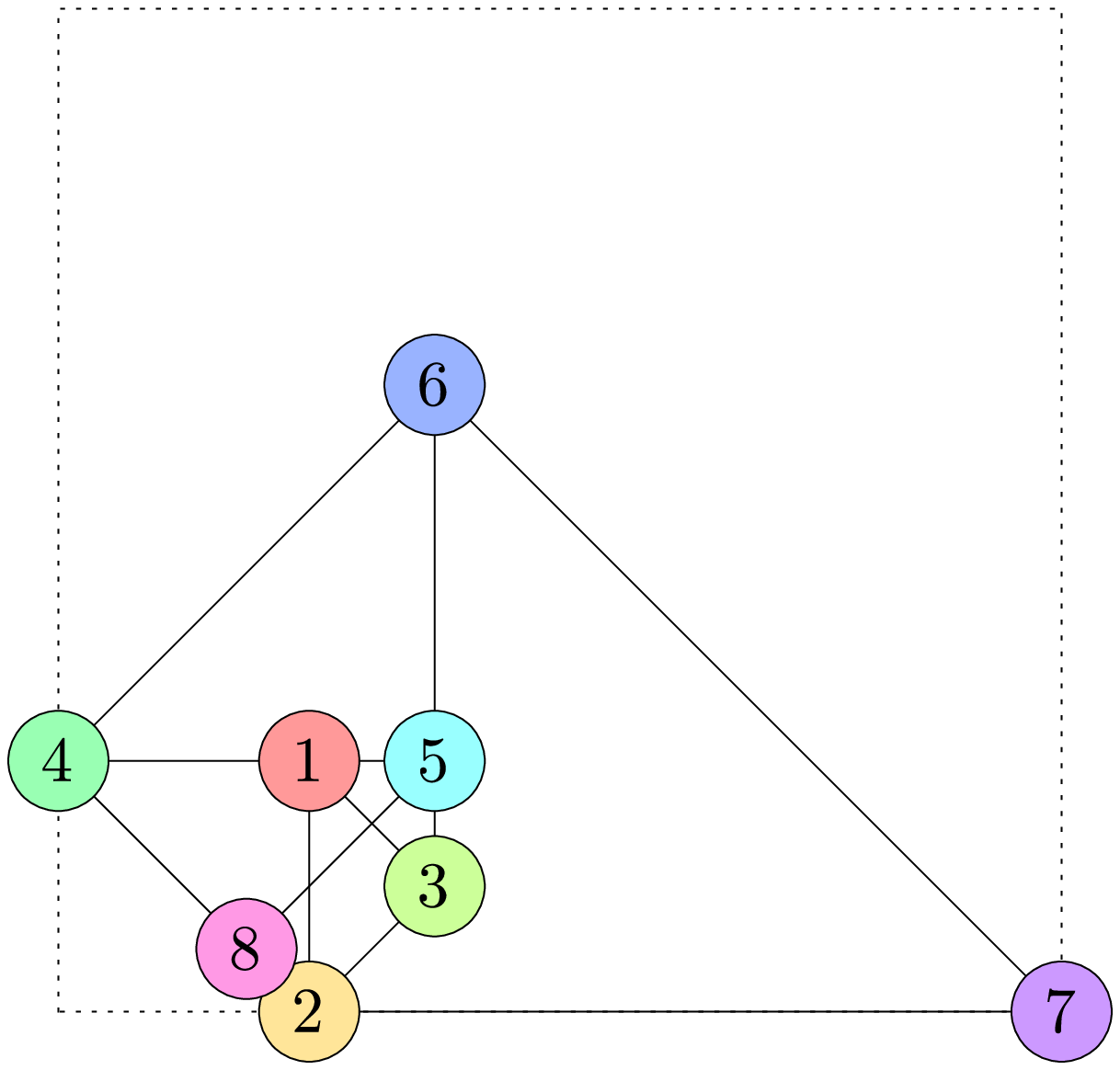}
\includegraphics[height=1.5in]{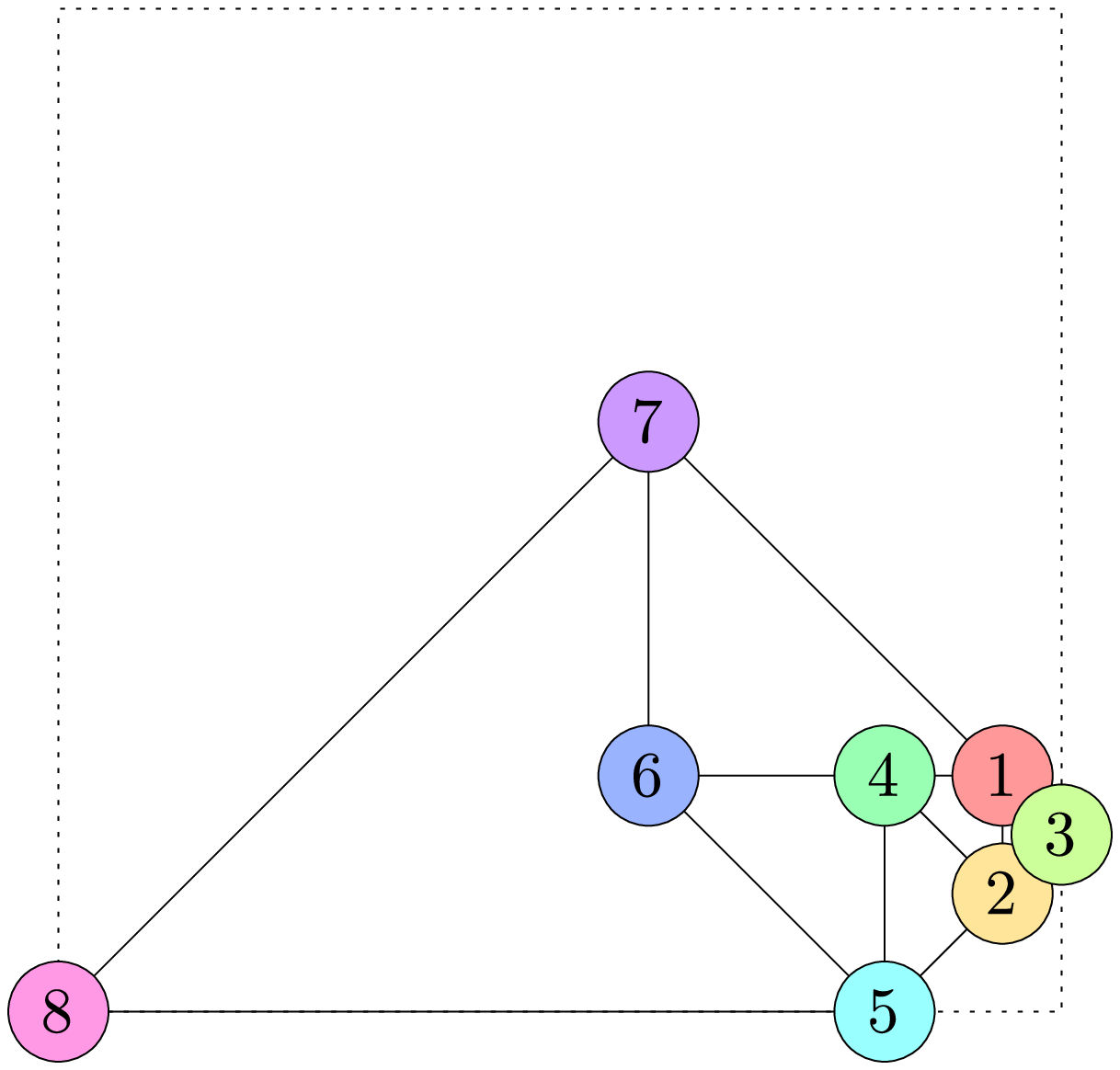}
\includegraphics[height=1.5in]{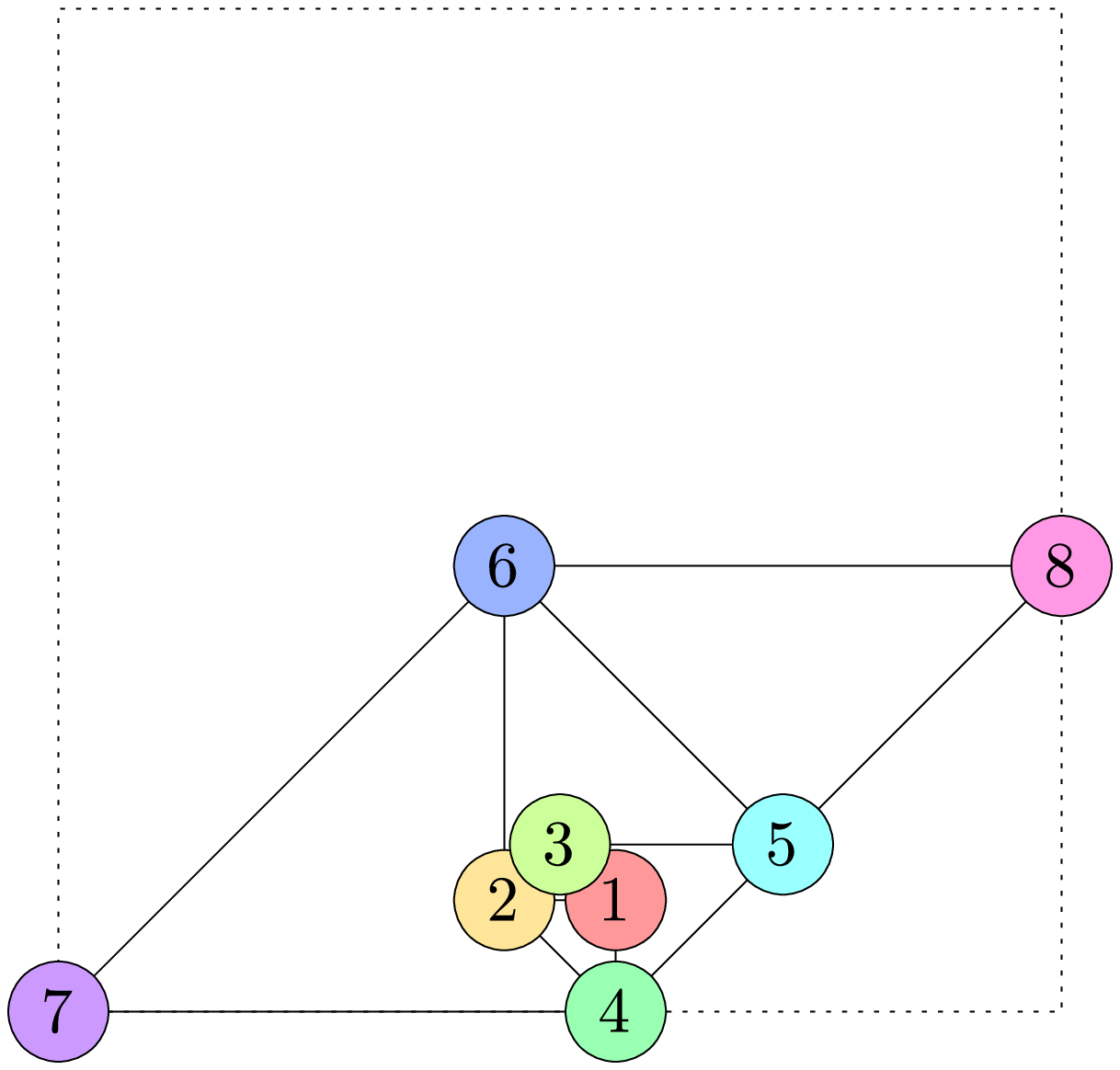}
\includegraphics[height=1.5in]{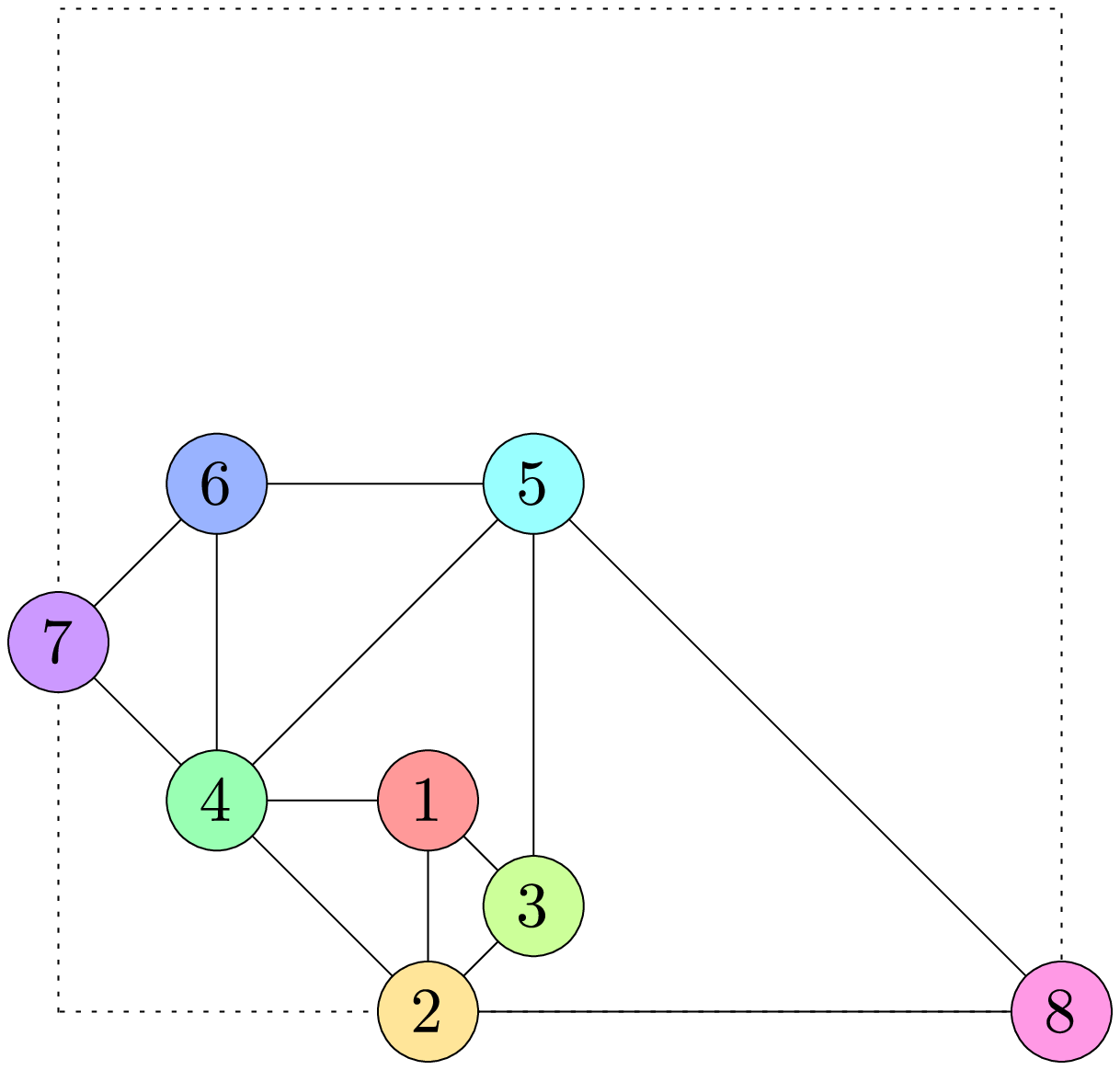}
\includegraphics[height=1.5in]{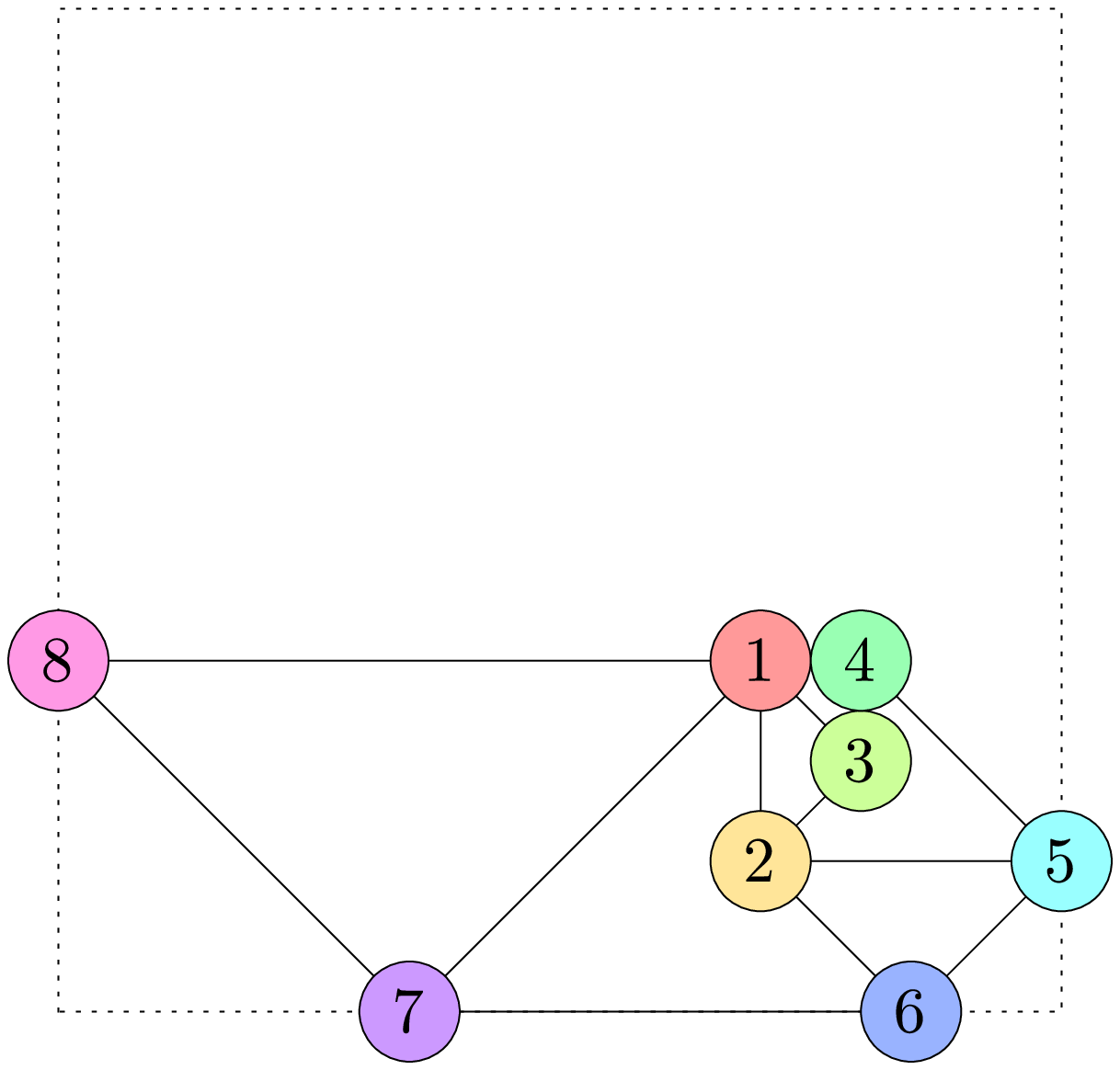}
\caption{Vertex configurations of eight queens with denominators 14 through 20.}
\label{fig:q8configs}
\end{figure}
\end{exam}

\Kot\ \cite[Section 1.1; p.\ 31 in 6th ed.]{ChMath} suggests a period expressed in terms of Fibonacci numbers.  

\begin{conj}[\Kot]\label{KotFib}
The counting quasipolynomial for $q$ queens has period $\lcm[F_q]$, the least common multiple of all positive integers up through the $q$th Fibonacci number $F_q$.
\end{conj}

The appearance of Fibonacci numbers here was the motivation for our study of Fibonacci configurations.  
\Kot's conjectured periods up to $q=7$ (from \cite[6th ed., p.\ 31]{ChMath}, or see them in our Part V) agree with this proposal and the theory of this section lends credence to it.  
(We discuss this in more detail in Part V.)  
We hope that our theory can yield a proof of the following (weaker) denominator analog of Conjecture \ref{KotFib}:

\begin{conj}\label{denomQ}
The denominator of the inside-out polytope for $q$ queens is exactly $\lcm[F_q]$, where $[F_q]=\{1,2,\ldots,F_q\}$. 
\end{conj}

\medskip
The discrete Fibonacci spiral is more complicated than the golden parallelogram.  
For each three-move piece $\pP$, the same linear transformation of the golden rectangle of $q$ semiqueens, independent of $q$, gives golden parallelogram configurations with the largest known denominator for $q$ copies of $\pP$.  
With four-move pieces, that is no longer true.  The linear transformation of the queen's Fibonacci spiral needed to get largest known denominators depends on $q$; in particular, its dilation factor increases as $q$ does.  
Consider the nightrider $\pN$.  In the least integral expansions of Figure~\ref{fig:n41} $\triangle\pN_1\pN_2\pN_3$ has width 72 in the first picture, $3\cdot72$ in the second, and $9\cdot72$ in the third.
Similarly, when we fit four nightriders in the next Fibonacci spiral, the smallest triangle dilates by a factor of four as each new piece is added, from width 128 to 512 to 2048.

We define a {\em twisted Fibonacci spiral} of $q$ pieces $\pP$ with moves $\{m_1,m_2,m_3,m_4\}$ by the move equations 
\[
\cH^{m_1}_{2i,2i+1},\quad \cH^{m_2}_{2i+1,2i+2},\quad \cH^{m_3}_{1,3},\quad \cH^{m_3}_{2i,2i+3},\quad \cH^{m_4}_{2i+1,2i+4}
\] 
for all $i$ such that both indices fall between $1$ and $q$, inclusive, and three fixations to ensure that the square box bounding all the pieces is as small as possible to make that all coordinates are integral.

By varying the choice of $m_1$, $m_2$, $m_3$, and $m_4$ we get different vertex configurations.  Consider nightriders in the following example.

\begin{exam}\label{ex:twisted}
The most obvious analog for nightriders of the queens' discrete Fibonacci spiral is that in Figure~\ref{fig:n41}, for which $m_1=1/2$, $m_2=-2/1$, $m_3=2/1$, and $m_4=-1/2$.  There is an alternate vertex configuration with larger denominator, the ``expanding kite'' shown in Figure~\ref{fig:n42}, which is a twisted Fibonacci spiral with $m_1=-2/1$, $m_2=1/2$, $m_3=2/1$, and $m_4=-1/2$.

\begin{figure}[htbp]
\includegraphics[height=1.7in]{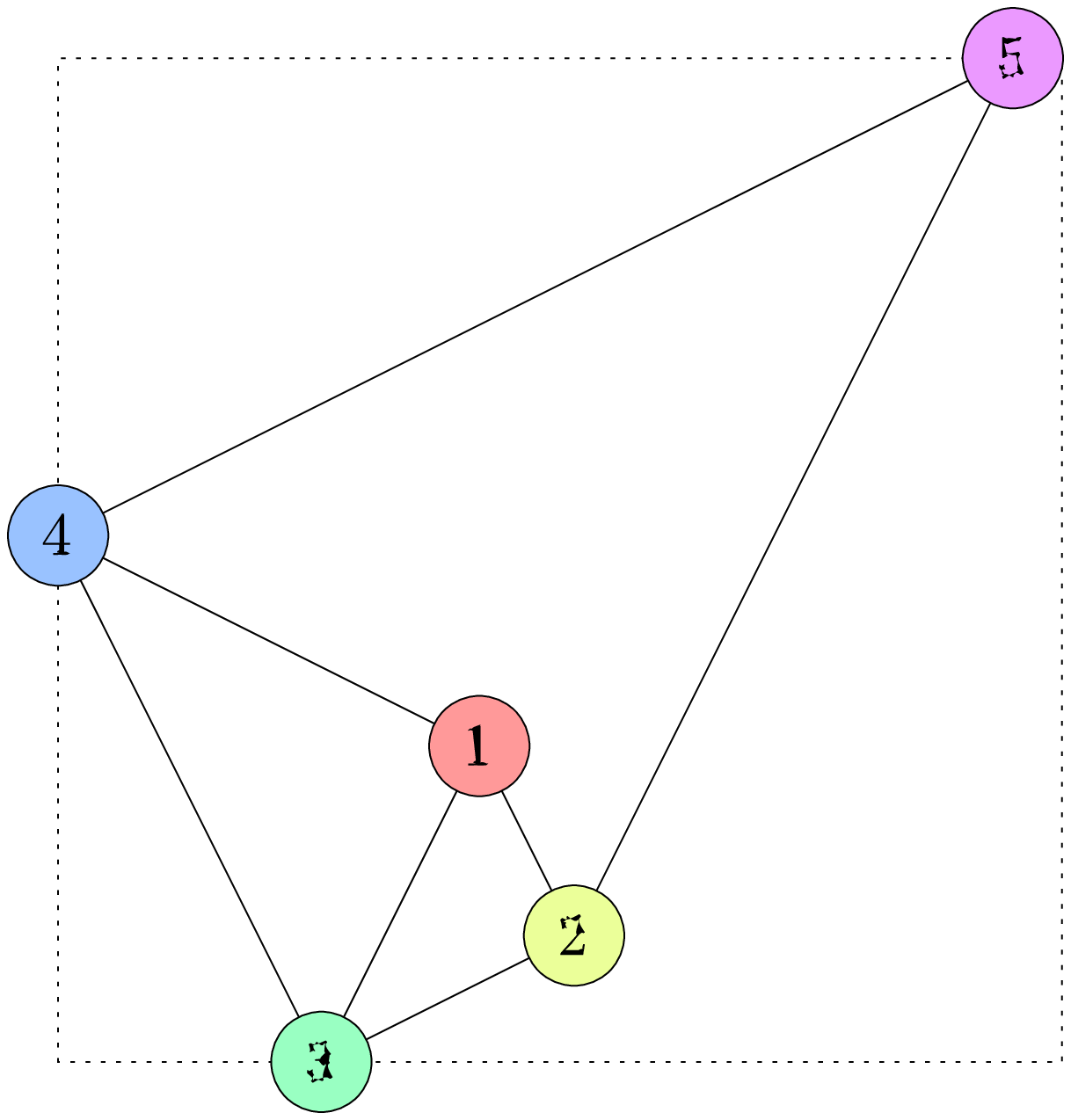}
\includegraphics[height=1.7in]{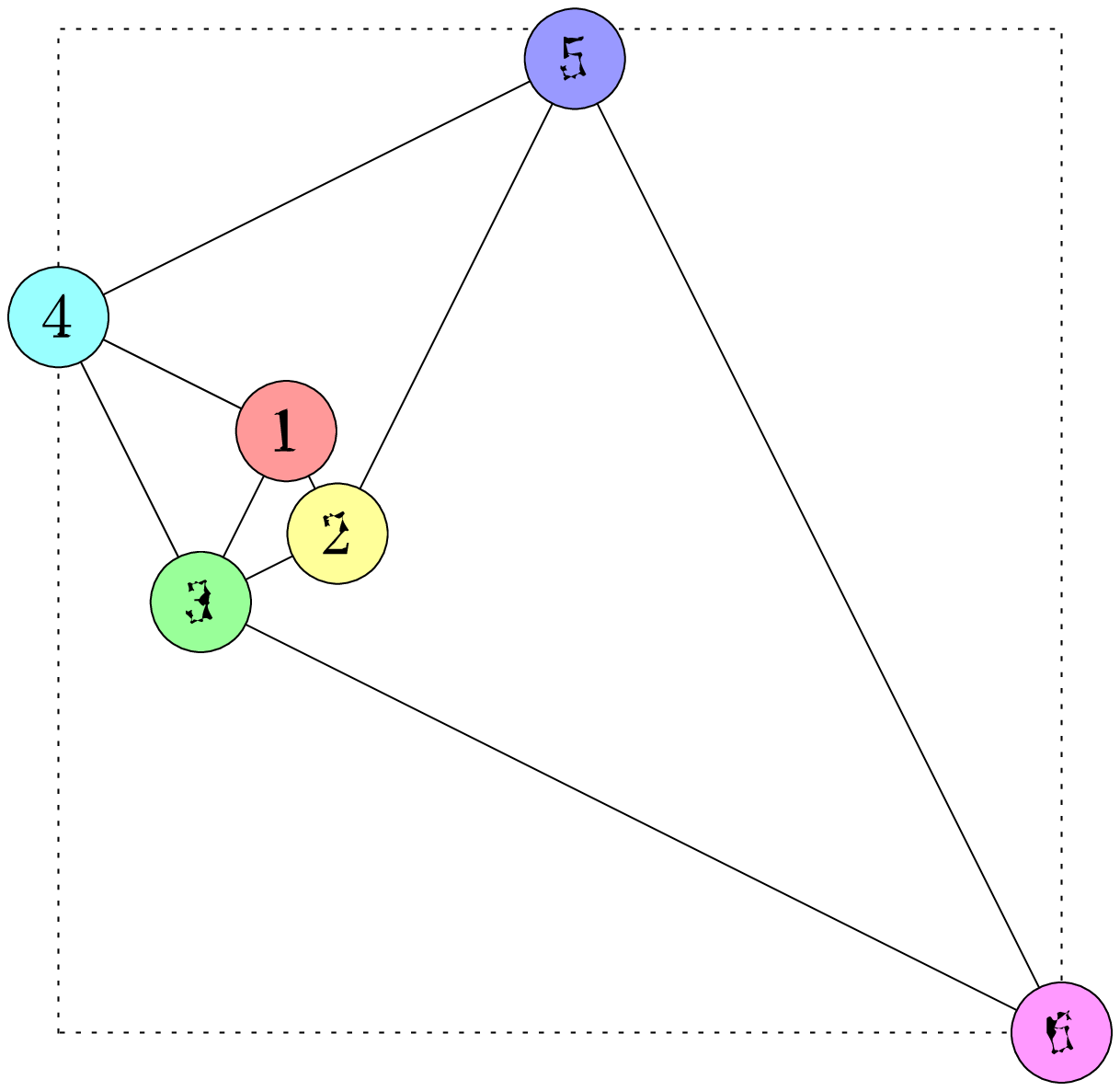}
\includegraphics[height=1.7in]{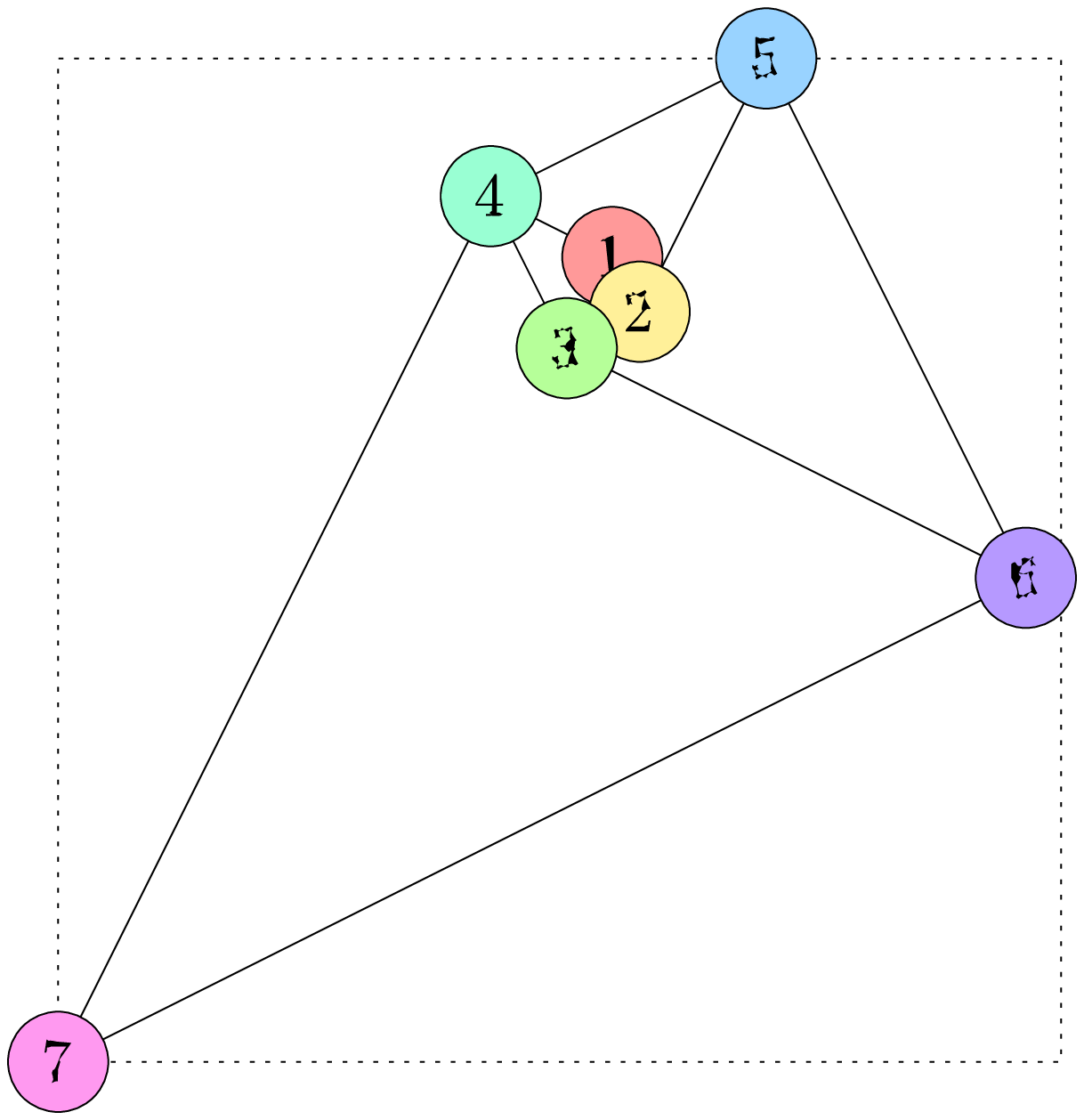}
\caption{A twisted Fibonacci spiral involving 5, 6, and then 7 nightriders.  Each configuration, in integral form, is contained in the next, expanded by a factor of 3.  The configurations in fractional form have denominators 286, 1585, and 8914.}
\label{fig:n41}
\end{figure}

\begin{figure}[htbp]
\includegraphics[height=1.7in]{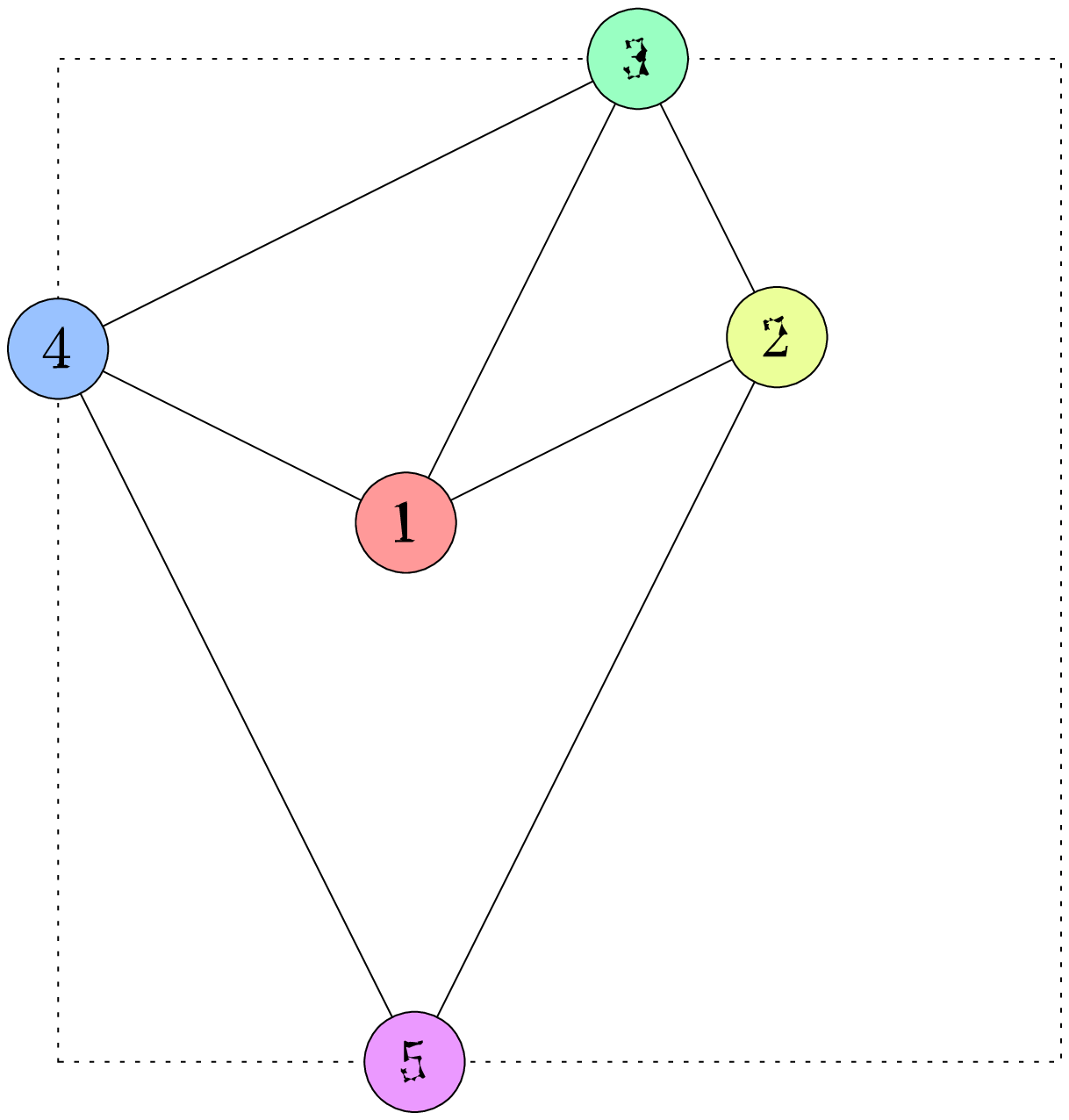}
\includegraphics[height=1.7in]{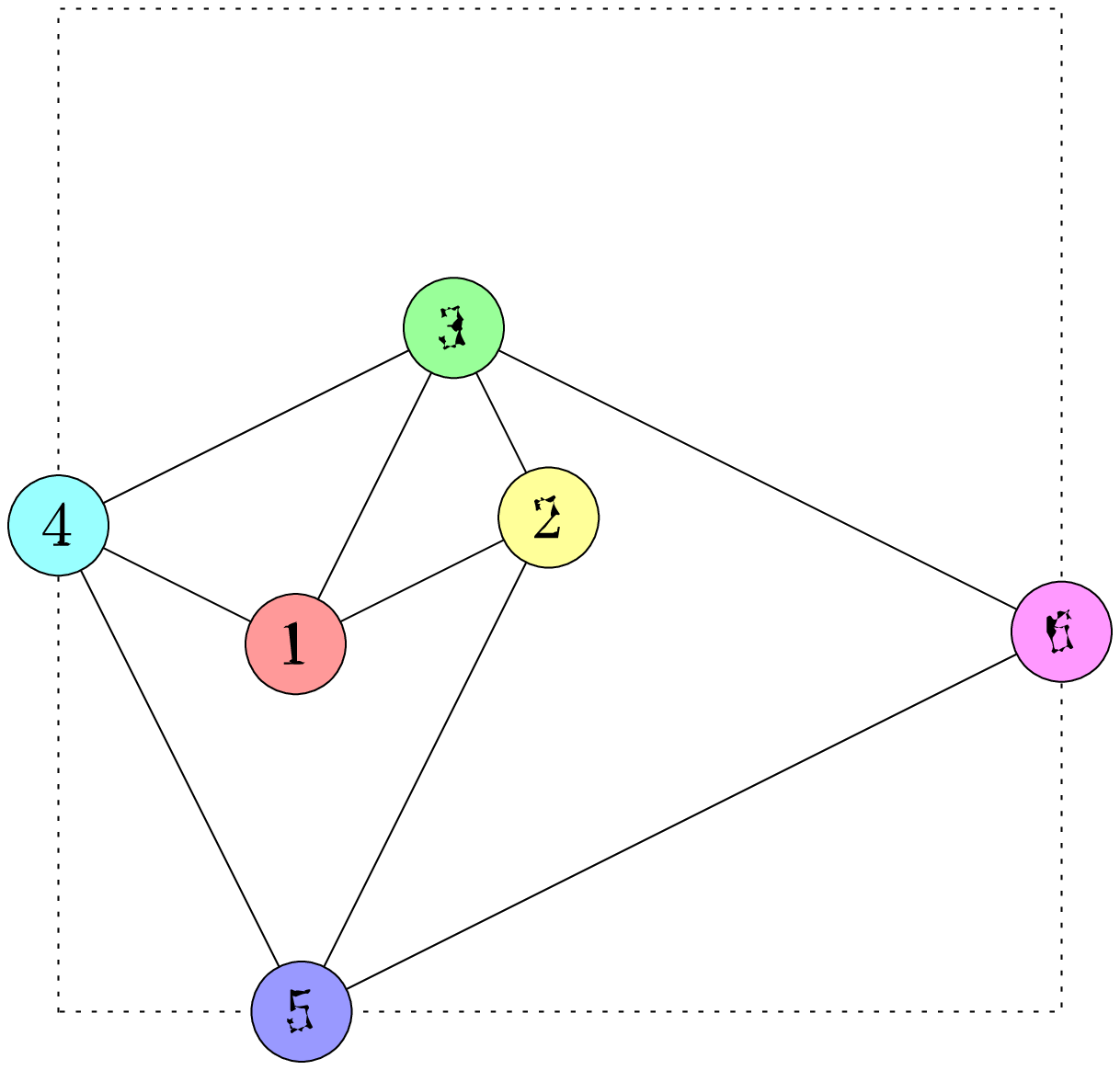}
\includegraphics[height=1.7in]{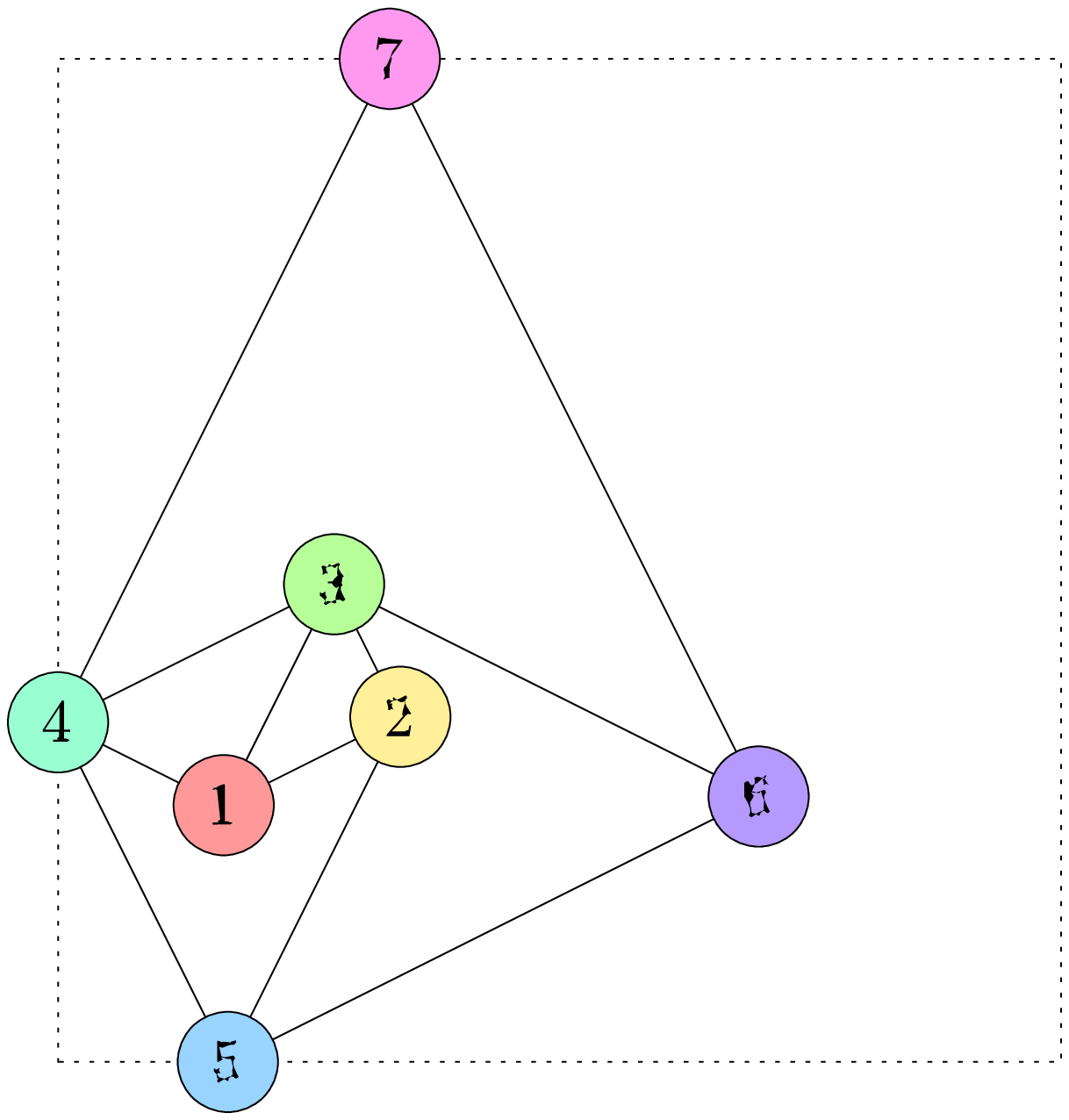}
\caption{The same expanding kite configuration, successively involving 5, 6, and then 7 nightriders.  In integral form, each configuration is contained in the next, expanded by a factor of 4.  In fractional form the configurations have denominators 346, 2030, and 11626.}
\label{fig:n42}
\end{figure}
\end{exam}

\begin{conj} 
For any piece $\pP$, there is a vertex configuration that maximizes the denominator and is a twisted Fibonacci spiral.
\end{conj}

Unlike for queens, the maximum denominator of a vertex for $q$ of a general piece $\pP$, call it $\Delta_q(\pP)$, is difficult to compute.  Furthermore, not every integer from $1$ to $\Delta_q$ may be a vertex denominator.  As an example, with three nightriders the only possible denominators are $1,2,3,4,5,10$.

\newpage
\section*{Dictionary of Notation}


\begin{enumerate}[]
\item $(c,d)$ -- coordinates of move vector (p.\ \pageref{d:mr})
\item $d/c$ -- slope of line or move (p.\ \pageref{d:slope-hyp})
\item $h$ -- \# of horizontal, vertical moves of partial queen (p.\ \pageref{d:partQ})
\item $k$ -- \# of diagonal moves of partial queen (p.\ \pageref{d:partQ})
\item $m = (c,d)$ -- basic move (p.\ \pageref{d:mr})
\item $m^\perp = (d,-c)$ -- orthogonal vector to move $m$ (p.\ \pageref{d:mrperp})
\item $n$  -- size of square board (p.\ \pageref{d:n})
\item $p$ -- period of quasipolynomial (p.\ \pageref{d:p})
\item $q$ -- \# of pieces on a board (p.\ \pageref{d:q})
\item $u_\pP(q;n)$ -- \# of nonattacking unlabelled configurations (p.\ \pageref{d:indistattacks})
\item $z=(x,y)$, $z_i=(x_i,y_i)$ -- piece position (p.\ \pageref{d:config})
\item $\bz=(z_1,\ldots,z_q)$ -- configuration (p.\ \pageref{d:config})
\item $\phi$ -- golden ratio $(1+\sqrt5)/2$ (p.\ \pageref{d:phi})
\end{enumerate}
\smallskip
%

\begin{enumerate}[]
\item $D$, $D_q(\pP)$ -- denominator of inside-out polytope (p.\ \pageref{d:D}) 
\item $F_q$ -- Fibonacci numbers ($F_0=F_1=1$) (p.\ \pageref{d:Fib})
\item $\M$ -- set of basic moves (p.\ \pageref{d:moveset})
\item $\Delta(\bz)$ -- denominator of vertex $\bz$ (p.\ \pageref{d:Deltabz}) 
\end{enumerate}
\smallskip
\begin{enumerate}[]
\item $\cA_\pP$ -- move arrangement of piece $\pP$ (p.\ \pageref{d:AP})
\item $\cB, \cB^\circ$ -- closed, open board: usually the square $[0,1]^2$, $(0,1)^2$ (p.\ \pageref{d:cB})
\item $\cE$ -- edge line of the board (p.\ \pageref{d:cE})
\item $\cH_{ij}^{d/c}$ -- hyperplane for move $(c,d)$ (p.\ \pageref{d:slope-hyp})
\item $\cP, \cP^\circ$ -- closed, open polytope (p.\ \pageref{d:cP})
\item $\cube, \ocube$ -- closed, open hypercube (p.\ \pageref{d:cP})
\item $(\cP,\cA_\pP), (\cube,\cA_\pP)$ -- inside-out polytope (p.\ \pageref{d:cP})
\item $(\cP^\circ,\cA_\pP), (\ocube,\cA_\pP)$ -- open inside-out polytope (p.\ \pageref{d:cP})
\item $\cX_{ij}:=\cH_{ij}^{1/0}$ -- hyperplane of equal $x$ coordinates (p.\ \pageref{d:XY})
\item $\cY_{ij}:=\cH_{ij}^{0/1}$ -- hyperplane of equal $y$ coordinates (p.\ \pageref{d:XY})
\end{enumerate}
\smallskip
\begin{enumerate}[]
\item $\bbR$ -- real numbers
\item $\bbR^{2q}$ -- configuration space (p.\ \pageref{d:configsp}) 
\item $\bbZ$ -- integers
\end{enumerate}
\smallskip
\begin{enumerate}[]
\item $\pN$ -- nightrider (p.\ \pageref{N})
\item $\pP$ -- piece (p.\ \pageref{d:P})
\item $\pQ^{hk}$ -- partial queen (p.\ \pageref{d:partQ})
\end{enumerate}

\newpage
\newcommand\otopu{\r{u}}

\end{document}